\documentclass[12pt]{amsart} 

\theoremstyle{plain}
\newtheorem{theorem}{Theorem}[section]
\newtheorem{lemma}[theorem]{Lemma}
\newtheorem{proposition}[theorem]{Proposition}
\newtheorem{corollary}[theorem]{Corollary}
\newtheorem{conjecture}[theorem]{Conjecture}

\theoremstyle{definition} 
\newtheorem{remark}[theorem]{Remark}

\usepackage{amssymb, amsthm, enumitem, comment, url, color}
\usepackage{amsmath}

\newcommand{\lz}{\left(}
\newcommand{\pz}{\right)}
\renewcommand{\epsilon}{\varepsilon}
\newcommand{\C}{\mathbb{C}}

\newcommand{\sumstar}{\sideset{}{^{*}}\sum}
\newcommand{\N}{\mathbb{N}}
\newcommand{\R}{\mathbb{R}}

\newcommand{\lab}{\left|}
\newcommand{\rab}{\right|}

\newcommand{\re}{\mathrm{Re}}

\newcommand{\im}{\mathrm{Im}}

\newcommand{\bfrac}[2]{\lz\frac{#1}{#2}\pz}

\renewcommand{\mod}[1]{\text{ (mod $#1$)}}
\newcommand{\odd}{\mathrm{\ odd}}

\renewcommand{\l}{\ell}
\renewcommand{\phi}{\varphi}
\newcommand{\leg}[2]{\left(\frac{#1}{#2}\right)}
\newcommand{\ch}{\mathrm{conv}}
\newcommand{\nonneg}{\mathrm{nonneg}}
\newcommand{\res}[1]{\underset{#1}{\mathrm{Res\ }}}

\begin{document}
		\title[Moments of real Dirichlet $L$-functions and MDS]{Moments of real Dirichlet $L$-functions and multiple Dirichlet series}
	
	\author{M. \v Cech}
	
	\address[Martin \v Cech]
	{Charles University, Faculty of Mathematics and Physics, Department of Mathematical Analysis and Department of Algebra, Sokolovsk\'a 83, 18600 Praha 8, Czech Republic
	}
	\email{martin.cech@matfyz.cuni.cz}
	
	\maketitle

	\begin{abstract}
		We consider the multiple Dirichlet series associated to the $k$th moment of real Dirichlet $L$-functions, and prove that it has a meromorphic continuation to a specific region in $\C^{k+1}$, which is conditional under the generalized Lindel\"of hypothesis for $k\geq 5$.
		
		As a corollary, we obtain asymptotic formulas for the first three moments with a power-saving error term, and detect the 0- and 1-swap terms in related problems for any $k$ (conditionally under the Generalized Lindel\"of Hypothesis), recovering the recent results of Conrey and Rodgers on long Dirichlet polynomials.
		
		The advantage of our method is its simplicity, since we don't need to modify the multiple Dirichlet series to obtain its meromorphic continuation. As a result, we obtain the asymptotic formulas directly in the form as they appear in the recipe predictions of Conrey, Farmer, Keating, Rubinstein and Snaith.
	\end{abstract}
	

	\section{Introduction and statement of results}
	The study of moments in families of $L$-functions is an important and active area of research in analytic number theory. In this paper, we will consider the moments in the family of real Dirichlet $L$-functions
	\begin{equation}\label{eqn: the studied moment}
		\sumstar_{d\geq1}f\bfrac {d}X L(s_1,\chi_{8d})\dots L(s_k,\chi_{8d}),
	\end{equation}
	where the sum runs over positive odd square-free integers, and $f$ is a smooth function with compact support in $(0,\infty)$. An asymptotic formula for these moments at the central point $s_1=\dots=s_k=1/2$ was first conjectured by Keating and Snaith \cite{KeSn} based on the similar behavior of values in families of $L$-functions and characteristic polynomials of random matrices. Later, Conrey, Farmer, Keating, Rubinstein and Snaith \cite{CFKRS} devised a purely number-theoretic heuristic, the so-called recipe, which provides the following more precise conjecture:
	
	\begin{conjecture}[Conrey, Farmer, Keating, Rubinstein, Snaith]\label{con: cfkrs}
		Let $X$ be large and $S=\{s_1,\dots,s_k\}$ be a set of complex numbers satisfying $|\re(s_j)-1/2|\ll 1/\log X$ and $|\im (s_j)|\ll X^{1-\epsilon}$. Then for some $\delta>0$,
		\begin{equation}\label{eqn: cfkrs conjecture main term}
			\begin{aligned}
				&\sumstar_{d\geq1}f\bfrac{d}{X}L(s_1,\chi_{8d})\dots L(s_k,\chi_{8d})\\
				&=\sum_{J\subset\{1,\dots,k\}}X^{1+\frac{|J|}{2}-\sum\limits_{j\in J}s_j} \tilde f\lz1+\frac{|J|}{2}-\sum_{j\in J}s_j\pz \prod_{j\in J}X(s_j)\cdot  T(S\setminus S_J\cup S_J^-)\\
				&+O(X^{1-\delta}),
			\end{aligned}
		\end{equation}
		where $S_J=\{s_j:j\in J\}$, $S_J^{-}=\{1-s_j :j\in J\}$, $\tilde f(s)$ denotes the Mellin transform of $f$,
		\begin{equation}\label{eqn: recipe term definition}
			T(S):=\frac{2}{3\zeta(2)}\sum_{\substack{n_1\dots n_k=\square,\\n_1\dots n_k\odd}}\frac{a( n_1\dots n_k)}{n_1^{s_1}\dots n_k^{s_k}}, \hspace{15pt} X(s):=\bfrac{\pi}{8}^{s-1/2}\frac{\Gamma\bfrac{1-s}{2}}{\Gamma\bfrac s2},
		\end{equation}
		and
		\begin{equation}\label{eqn: a_n definition}
			a(n)=\prod_{p|n}\lz1+\frac{1}p\pz^{-1}.
		\end{equation}
	\end{conjecture}
	\begin{remark}
		$T(S)$ in \eqref{eqn: recipe term definition} should be understood as the meromorphic continuation of the Dirichlet series on the right-hand side. From \cite[Lemma 1 and Section 5.1]{CoRo}, we have
		\begin{equation}\label{key}
			T(S)=\frac{2}{3\zeta(2)}\prod_{j=1}^k\zeta(2s_j)\prod_{\substack{j_1,j_2=1\\j_1\neq j_2}}^k\zeta(s_1+s_2)\cdot E(S),
		\end{equation} where $E(S)$ is an Euler product that is absolutely convergent if $\re(s_j)>1/4$ for each $j=1,\dots,k$. Therefore the right-hand side of \eqref{eqn: cfkrs conjecture main term} is defined at least if $1/4<\re(s_j)<3/4$ for each $j=1,\dots,k$. In fact, our Theorem \ref{thm: main theorem} implies that we have a meromorphic continuation to a larger region.
		
		The individual terms $T(S\setminus S_J\cup S_J^-)$ have some poles, for instance if $s_j=1/2$ for some $j$, but the whole right-hand side of \eqref{eqn: cfkrs conjecture main term} is analytic in the stated region. 
	\end{remark}
	
	The first results towards Conjecture \ref{con: cfkrs} are due to Jutila \cite{Jut}, who considered the moments at the central point $s_1=\dots=s_k=1/2$ and obtained an asymptotic formula with power-saving error term for $k=1$, and the main term in the asymptotic formula for $k=2$. Soundararajan \cite{Sou} then obtained the asymptotic formula for $k=2,3$ with a power-saving error term. These remain the highest known moments for which an asymptotic formula with a power-saving error term is known.
	
	In the past years, some progress in the case $k=4$ was made. Based on the work of Soundararajan and Young \cite{SoYo}, Shen \cite{She} obtained the main term in this case under the generalized Riemann hypothesis (GRH). Recently, Li \cite{Li} managed to obtain a similar result as \cite{SoYo} unconditionally, and it was extended to the case of real Dirichlet $L$-functions by Shen and Stucky \cite{ShSt}.
	
	\smallskip
	
	To put the results into further context, let us recall some terminology. On the right-hand side of \eqref{eqn: cfkrs conjecture main term}, the terms with $|J|=\l$ are called the $\l$-swap terms, with the 0-swap term also being the diagonal one. Therefore there is one diagonal term of size $X$, several 1-swap terms of size $X^{3/2-s_j}$, 2-swap terms of size $X^{2-s_{j_1}-s_{j_2}}$, etc. Note that they all have asymptotically similar size if $\re(s_j)$ are as in Conjecture \ref{con: cfkrs}. The above mentioned work of Jutila may be considered as detecting the diagonal terms, while the application of Poisson summation by Soundararajan enabled him to also detect the 1-swap terms. It shows up that in many related problems, including the study of other families of $L$-functions, we are able to detect the 0- and 1-swap terms, but no others. Examples include the study of moments to the right of the critical line (see \cite[Theorem 1]{GoHo}), or the work of Conrey and Rodgers \cite{CoRo} on long Dirichlet polynomials, who are able to detect the 0- and 1-swap terms assuming the Generalized Lindel\"of hypothesis (GLH, see \eqref{eqn: GLH for individual L-function}). In the approaches based on Poisson summation, the 1-swap terms often don't appear in the conjectured form, but some involved calculations are needed to see that the obtained result and the prediction agree (see for instance \cite[Section 5]{CoRo}).
	
	\smallskip
	
	In this paper, we consider the multiple Dirichlet series (MDS) approach towards Conjecture \ref{con: cfkrs}.
	
	Using Perron's formula, the moment in \eqref{eqn: the studied moment} equals
	\begin{equation}\label{eqn: moment as integral}
		\frac{1}{2\pi i}\int_{(c)}A(s_1,\dots,s_k,w) \tilde f(w)X^wdw,
	\end{equation} where
	\begin{equation}\label{eqn: A definition}
		A(s_1,\dots,s_k,w)=\sumstar_{d\geq 1}\frac{L(s_1,\chi_{8d})\dots L(s_k,\chi_{8d})}{d^w}
	\end{equation} is the associated multiple Dirichlet series. If we had sufficient information about $A(s_1,\dots,s_k,w)$, we could shift the integral in \eqref{eqn: moment as integral} to the left, capture the contribution of the residues and obtain an asymptotic formula for the moment studied. This strategy was first applied by Goldfeld and Hoffstein \cite{GoHo} and further developed by Diaconu, Goldfeld and Hoffstein \cite{DGH}, who proved that it leads to a similar result as that predicted by Keating and Snaith for our family, provided $A(s_1,\dots,s_k,w)$ has a meromorphic continuation to a region in $\C^{k+1}$ that contains the point $(1/2,\dots,1/2,1)$. They were able to obtain such a continuation for $k\leq 3$, thus recovering the results of Jutila and Soundararajan, with a better error term. They also conjectured the appearance of secondary terms of size $X^{3/4}$ in the $k$th moment for $k\geq3$, which were not predicted by other methods, and their existence for $k=3$ was lately confirmed by Diaconu and Whitehead \cite{DiWh}.
	
	In forthcoming work of Baluyot and the author \cite{BaCe}, we will extend the results of Diaconu, Goldfeld and Hoffstein and establish a connection between the recipe prediction and the multiple Dirichlet series approach, namely that the terms that arise on the right-hand side of \eqref{eqn: cfkrs conjecture main term} are in a 1-1 correspondence with the poles of $A(s_1,\dots,s_k,w)$ and their residues, and this phenomenon holds for many other families of $L$-functions. This was already observed by the author in \cite{Cec1} and \cite{Cec2}.
	
	Therefore to prove Conjecture \ref{con: cfkrs}, the remaining essential problem is obtaining a meromorphic continuation of $A(s_1,\dots,s_k,w)$ to a sufficiently large region.
	
	\smallskip	
	In most works in multiple Dirichlet series, such as \cite{DGH}, the strategy to meromorphically continue $A(s_1,\dots,s_k,w)$ is to replace it by a modified multiple Dirichlet series $Z(s_1,\dots,s_k,w)$ by inserting some carefully chosen weights into its definition, such that $Z(s_1,\dots,s_k,w)$ satisfies certain group of functional equations. The drawback of this approach is that it is very difficult to find these weights, and extra work is needed to go back to $A(s_1,\dots,s_k,w)$ once the analytic properties for $Z(s_1,\dots,s_k,w)$ are established. Our goal is to circumvent these disadvantages.
	
	In this paper, we work directly with $A(s_1,\dots,s_k,w)$. Assuming GLH, we are able to obtain a region of meromorphic continuation for each $k\geq 1$. Our results are unconditional for $k\leq 4$, since we only need ``GLH on average'', which is provided by an estimate of Heath-Brown for the fourth moment \eqref{eqn: Heath-Brown fourth moment}. The region we obtain is large enough to enable us to compute the moments for $k=1,2$ and 3 unconditionally, and detect the 1-swap terms in various situations for any $k$ under GLH.
	
	An advantage of working directly with $A(s_1,\dots,s_k,w)$ is, in contrast with many other classical works, that the computations are very simple, and the asymptotic formulas naturally come in the same form as predicted by the recipe in Conjecture \ref{con: cfkrs}.
	
	Our main tool in obtaining the meromorphic continuation of $A(s_1,\dots,s_k, w)$ is the functional equation valid for $L$-functions of all (not necessarily primitive) Dirichlet characters \eqref{eqn: fe for all L-functions}. This is a direct analogue of the Poisson summation formula of Soundararajan \cite[Lemma 2.6]{Sou} used to obtain the third moment or detect the 1-swap terms in the classical setting, so it is not surprising that both techniques lead to similar results. 
	
	A similar strategy was applied by the author in \cite{Cec1}, and developed in other directions by Gao and Zhao \cite{GaZh1}--\cite{GaZh5}.
	
	\smallskip
	
	We now state the main result. For $J\subset\{1,\dots,k\},$ let $\sigma_J:\C^{n+1}\rightarrow\C^{n+1}$ be the affine map
	\begin{equation}\label{key}
		\sigma_J: (s_1,\dots,s_k,w)\mapsto\lz s_1^J,\dots,s_k^J,w+\sum_{j\in J}s_j-\frac{|J|}{2}\pz,
	\end{equation}  where
	\begin{equation}\label{key}
		s_j^J=\begin{cases}
			1-s_j,&\hbox{if $j\in J$},\\
			s_j,&\hbox{if $j\notin J$}.
		\end{cases}
	\end{equation}
	Note that $\sigma_J$ are pairwise commutative involutions, and that $\sigma_{\{j\}}\sigma_{\{i\}}=\sigma_{\{i,j\}}$ for any $1\leq i<j\leq k$.
	We will see in \eqref{eqn: A fe} that there is a functional equation relating $A(s_1,\dots,s_k,w)$ and $A(\sigma_J(s_1,\dots,s_k,w))$.
	
	We denote by $X(s)$ the ratio of the gamma factors that appear in the functional equations of $L(s,\chi_{8d}):$
	\begin{equation}\label{key}
		X(s):=\bfrac{\pi}{8}^{s-1/2}\frac{\Gamma\bfrac{1-s}{2}}{\Gamma\bfrac s2}.
	\end{equation}
	Throughout the paper, $\epsilon$ denotes an arbitrarily small positive number, not necessarily the same at each appearance. All implied constants are allowed to depend on $\epsilon$. Moreover, $f(u)$ denotes a smooth, function that is compactly supported in $(0,\infty)$, and 
	\begin{equation}\label{key}
		\tilde f(s):=\int_{0}^\infty f(u)u^{s-1} du
	\end{equation} denotes the Mellin transform of $f$. Since $f$ is smooth and compactly supported, $\tilde f(s)$ is defined for every $s\in\C$, and it decays faster than any polynomial in vertical strips.
	
	\begin{theorem}\label{thm: main theorem}
		Assume that $k\leq 4$ or that GLH holds. Then $A(s_1,\dots,s_k,w)$ has a meromorphic continuation to the region which is the intersection of the half-spaces
		\begin{equation}\label{eqn: hyperplanes for A}
			\begin{aligned}
				\re(w)&>1/2,\\
				\re(s_i+2w)&>7/4,\hspace{20pt} i\in\{1,\dots,k\},\\
				\re(s_{i_1}+s_{i_2}+2w)&>5/2,\hspace{20pt} \hbox{$i_1,i_2\in\{1,\dots,k\}$,}\\
				\re(s_{i_1}+s_{i_2}+s_{i_3}+2w)&>13/4,\hspace{14pt} \hbox{$i_1,i_2,i_3\in\{1,\dots,k\}$}\\
				\re(s_{i_1}+s_{i_2}+s_{i_3}+s_{i_4}+2w)&>4,\hspace{32pt} \hbox{$i_1,i_2,i_3,i_4\in\{1,\dots,k\}$,}
			\end{aligned}
		\end{equation} (with $i_1,i_2,i_3,i_4$ pairwise different) and their reflections under the transformations $\sigma_J$, $J\subset\{1,\dots,k\}$.
		
		Its only poles in this region are at the points
		\begin{equation}\label{key}
			w=1+\frac{|J|}{2}-\sum_{j\in J}s_j,
		\end{equation} where $J\subset\{1,\dots,k\}$, with residues
		\begin{equation}\label{eqn: main theorem residues}
			\res{w=1+\frac{|J|}{2}-\sum_{j\in J}s_j}A(s_1,\dots,s_k,w)=\prod_{j\in J}X(s_j) \cdot T(S\setminus S_J\cup S_J^-),
		\end{equation} where $T(S\setminus S_J\cup S_J^-)$ is as in Conjecture \ref{con: cfkrs}.
		
		Moreover, $A(s_1,\dots,s_k,w)$ is polynomially bounded in vertical strips. In particular, for any fixed $\sigma_1,\dots,\sigma_k,\omega\in\R$, we have away from the poles
		\begin{equation}\label{eqn: polynomial boundedness def in theorem}
			|A(\sigma_1+it_1,\dots,\sigma_k+it_k,\omega+it_{k+1})|\ll \prod_{j=1}^{k+1}(1+|t_{j}|)^{C}
		\end{equation} for some constant $C$.
	\end{theorem}
	
	In the region \eqref{eqn: hyperplanes for A}, the conditions with $s_{j_1},\dots,s_{j_t}$ appear only if $k\geq t$.
	\begin{remark}\label{remark: reflections of hyperplanes}
		\begin{enumerate}
			\item For $j\in\{1,\dots,k\}$, the reflection under $\sigma_{\{j\}}$ of the half-space $\re(w)>1/2$ is $\re(w+s_j)>1$. For $1\leq t\leq 4$, the half-space $\re(s_{i_1}+\dots+s_{i_t}+2w)>\frac{4+3t}{4}$ is invariant under $\sigma_{\{j\}}$ if $j\in\{i_1,\dots,i_t\},$ and it is \begin{equation}
				\re(s_{i_1}+\dots+s_{i_t}+2s_j+2w)>\frac{4+3t}{4}+1
			\end{equation} otherwise. 
			
			Therefore each half-space in the region in Theorem \ref{thm: main theorem} is given by
			\begin{equation}\label{eqn: higher level type 0 hyperplanes}
				\re(s_{j_1}+\dots+s_{j_\l}+w)>\frac{1+\l}{2},
			\end{equation} or
			\begin{equation}\label{eqn: higher level type t hyperplanes}
				\re(s_{i_1}+\dots+s_{i_t}+2s_{j_1}+\dots+2s_{j_\l}+2w)>\frac{1+3t}{4}+\l,
			\end{equation}
			where $1\leq t\leq 4$, $0\leq\l\leq k-4$, and $s_{i_1},\dots,s_{i_t},s_{j_1},\dots,s_{j_\l}$ are pair-wise different.
			\item In order to prove Conjecture \ref{con: cfkrs} for some $k$, we need continuation past $w=1$ when $s_1=\dots=s_k=1/2$. Theorem \ref{thm: main theorem} allows us to achieve this for $k\leq 3$. For $k\geq 4$, the last condition prevents us from shifting the integral far enough. Note that if $\re(s_j)\geq 1/2$ for each $j=1,\dots,k,$ the conditions in \eqref{eqn: higher level type 0 hyperplanes} and \eqref{eqn: higher level type t hyperplanes} with $\l>0$ don't limit how far the $w$-integral can be shifted.
			\item The description of the region may seem complicated, but is somewhat natural. In particular, the symmetry under $\sigma_J$ follows from the functional equations of $A(s_1,\dots,s_k,w)$, the first condition in \eqref{eqn: hyperplanes for A} essentially comes from our treatment of the square-freeness condition in the sum, the next 3 conditions determine the error term that can be obtained in the cases $k=1,2,3$ (and can be improved with more careful treatment), while the last condition is the obstruction that prevents us from computing the moments for $k\geq 4$ or detecting the 2-swap terms (see also the remarks under Corollary \ref{cor: long Dirichlet polynomials}).
		\end{enumerate}
	\end{remark}
	Theorem \ref{thm: main theorem} is very flexible, as it implies many results on moments and Dirichlet polynomials after a straightforward application of Perron's formula. An essential part in obtaining the theorem in such generality is the determination of the convex hull of the union of a region of meromorphic continuation of $A(s_1,\dots,s_k,w)$ and its transformations by $\sigma_J$, which is achieved in Section \ref{sec: convex hull}.
	
	Note that the residues in \eqref{eqn: main theorem residues} resemble part of the main terms in Conjecture \ref{con: cfkrs}. We will explicitly compute the residue at $w=1$ and the other poles will appear as a consequence of the functional equations.
	
	\smallskip
	
	As a first corollary, we prove an asymptotic formula for the first three moments with a power-saving error term.
	\begin{corollary}\label{cor: first three moments}
		Let $k=1,2$, or $3$, and fix $s_1,\dots,s_k$ with $\re(s_j)\geq 1/2$. Then we have
		\begin{equation}\label{key}
			\begin{aligned}
				&\sumstar_{d\geq 1}f\bfrac{d}{X}L(s_1,\chi_{8d})\dots L(s_k,\chi_{8d})\\
				&=\sum_{J\subset\{1,\dots,k\}}\prod_{j\in J}X(s_j)\cdot X^{1+\frac{|J|}{2}-\sum\limits_{j\in J}s_j}\tilde f\lz1+\frac{|J|}{2}-\sum_{j\in J}s_j\pz T(S\setminus S_J\cup S_J^-)\\
				&+O(X^{\max\{1/2,e_k\}+\epsilon}),
			\end{aligned}
		\end{equation}
		where
		\begin{equation}\label{key}
			e_k=\frac{4+3k}{8}-\sum_{j=1}^k\frac{\re(s_j)}{2}.
		\end{equation}
	\end{corollary}
	At the central point $s_j=1/2$, the exponent $e_k$ for $k=1,2,3$ in the error term above is respectively $5/8, 3/4, 7/8$. We didn't aim for optimizing the error term, it is plausible that a more careful treatment of these cases would lead to improvements. 
	
	In the case $k=1$, the essentially best possible error term (which is $X^{1/4+\epsilon}$ under GRH or $X^{1/2+\epsilon}$ unconditionally) was recently obtained by the author in \cite{first_moment} using the same method, but with a much more careful treatment of the square-freeness condition. 
	
	A similar technique should lead to an improvement for $k=2$ as well, where the essentially optimal error term $X^{1/2+\epsilon}$ is known due to Sono \cite{Sono}.

	For $k=3$, it might be possible to obtain the exponent $3/4$, first proved by Young \cite{You2}. However, due to the presence of the lower order terms in the third moment proved in \cite{DiWh}, it is impossible to obtain a better error without modifying our method to detect the lower order terms as well, which would require a possibility to iterate the functional equations in $s_j$ and $w$. 
	
	We hope to address these questions in a later paper.
	
	A similar conclusion holds for $s_j$ with $\re(s_j)<1/2$, except that one needs to take into account the conditions in \eqref{eqn: higher level type 0 hyperplanes} and \eqref{eqn: higher level type t hyperplanes} with $\l\geq 1$. Alternatively, it is possible to apply the functional equation to translate into the case $\re(s_j)\geq 1/2$.
	
	We prove Corollary \ref{cor: first three moments} for fixed $s_1,\dots,s_k$, but it would be straightforward to obtain a uniform result, at least in the range $\im(s_j)\ll X^{\delta}$ for some small $\delta$ (depending on $C$ in \eqref{eqn: polynomial boundedness def in theorem}). We believe that one should be able to take any $\delta>0$, so that Conjecture \ref{con: cfkrs} holds in a greater uniformity than stated, similarly as in \cite[Theorem 1.2]{Cec1}. See Remark \ref{remark: possible uniformity} for further discussion. It would also be possible to obtain similar results for the moments twisted by a fixed real Dirichlet character. However, it is again unclear what the optimal result would be when it comes to uniformity in the conductor of the twisting character.
	The same also applies to Corollary \ref{cor: moments to the right of critical line} and Corollary \ref{cor: long Dirichlet polynomials}.
	
	\smallskip
	
	As a second corollary, we consider moments to the right of the critical line. In this case, under GLH, these moments can be computed and they are dominated by the diagonal terms. We are able to detect the 1-swap terms as well (note that these are smaller by a power of $X$ if $\re(s_j)>1/2$), with a power saving error term.
	
	\begin{corollary}\label{cor: moments to the right of critical line}
		Assume that $k\leq 4$ or that GLH holds. Let $0<\delta<1/4$, $\nu>0$, and let $s_1,\dots,s_k\in\C$ be fixed with $\re(s_j)\in(1/2+\delta,1/2+(2-\nu)\delta)$ for each $j=1,\dots,k.$ Then
		\begin{equation}\label{eqn: detecting 1-swaps to the right}
			\begin{aligned}
				\sumstar_{d\geq1}&f\bfrac{d}{X}L(s_1,\chi_{8d})\dots L(s_k,\chi_{8d})=X\tilde f(1)T(S)\\
				&+\sum_{j=1}^k X^{3/2-s_j}\tilde f(3/2-s_j)T(S\setminus\{s_j\}\cup\{1-s_j\})+O(X^{1-2\delta+\epsilon}).
			\end{aligned}
		\end{equation}
	\end{corollary}
	Note that in \eqref{eqn: detecting 1-swaps to the right}, the terms of size $X^{3/2-s_j}$ are the terms with $|J|=1$ in \eqref{eqn: cfkrs conjecture main term}, and they are $\gg X^{1-2\delta+\nu\delta}$ in our range, so we have a power-saving error term.
	
	A similar result can be proved if $\re(s_j)<1/2$, for instance by an application of the functional equation. Note that in this case, the moment will no longer be dominated by the diagonal terms, but by a higher swap term. In particular, if $s_{j_1},\dots,s_{j_\l}$ is the set of all variables with real part $<1/2$, the largest term will be the $\l$-swap term of size $X^{1+\frac{\l}{2}-\sum\limits_{1\leq n\leq \l}s_{j_n}}$.
	
	\smallskip
	
	Unconditionally, one can detect the diagonal terms in all moments provided the real parts are not too far in the critical strip (see for instance \cite[Theorem 7.7]{Tit}). It would be interesting to see what can be proved unconditionally using our method and whether it would be possible to recover some higher swap terms in this case.
	
	\smallskip
	
	Another possible approach to Conjecture \ref{con: cfkrs} is to approximate the L-functions by long Dirichlet polynomials. Based on the work of Conrey and Keating \cite{CoKe1}--\cite{CoKe5}, there has been a considerable amount of work in this direction in the recent years for different families of $L$-functions (\cite{BaTu, CoFa, CoRo, HaNg}). In this setting, the higher swap terms should appear with increasing length of the polynomial. Current methods usually enable us to detect the 0- and 1-swap terms in various families, but no higher.
	
	The family of real Dirichlet characters was treated by Conrey and Rodgers \cite{CoRo}. They consider the averaged Dirichlet polynomial of the form
	\begin{equation}\label{key}
		\sumstar_{d\geq 1}f(d/X)\sum_{n_1,\dots,n_k\geq 1} \frac{\chi_{8d}(n_1\dots n_k)}{n_1^{s_1}\dots n_k^{s_k}}W\bfrac{n_1\dots n_k}{N},
	\end{equation} where $W$ is a smooth, compactly supported test function, $N=X^\eta$, and $\re(s_j)\approx 1/2$. If  $\eta<1$, only the diagonal terms contribute in the asymptotic formulas, while for $1\leq \eta<2$, Conrey and Rodgers obtain an asymptotic formula with the diagonal and 1-swap terms. In general, it is conjectured that the $\l$-swap terms should appear for $\eta\geq\l$ (and this is more-or-less equivalent to Conjecture \ref{con: cfkrs}). We remark that Conrey and Rodgers consider the more general moments twisted by some fixed real Dirichlet characters and obtain a result that is uniform in the conductor of the twisting character. It is straightforward to modify our multiple Dirichlet series to contain the twist and obtain an analogue of Theorem \ref{thm: main theorem}, but obtaining a dependence of the error term on the twisting character causes some difficulties similar to determining the uniformity in the imaginary part of the $s$-variables (see also Remark \ref{remark: possible uniformity}).
	
	The following corollary is analogous to \cite[Theorem 1]{CoRo} (with $\l=1$ and similar error term).
	\begin{corollary}\label{cor: long Dirichlet polynomials}
		Assume that $k\leq 4$ or that GLH holds. Let $a$ be a small positive constant (say $a=1/10$), and assume that $0<\re(s_j)-1/2\ll 1/\log X.$ Then we have
		\begin{equation}\label{key}
			\begin{aligned}
				&\sumstar_{d\geq 1}f(d/X)\sum_{n_1,\dots,n_k\geq 1} \frac{\chi_{8d}(n_1\dots n_k)}{n_1^{s_1}\dots n_k^{s_k}}W\bfrac{n_1\dots n_k}{N}
				=\frac{1}{2\pi i}\int_{(a)}\tilde W(s)N^s\\
				&\times\sum_{\substack{J\subset\{1,\dots,k\}\\|J|\leq 1}}X^{1+\frac{|J|}{2}-\sum\limits_{j\in J}(s_j+s)}\tilde f\lz1+\frac{|J|}{2}-\sum_{j\in J}(s_j+s)\pz T_s(S\setminus S_J\cup S_J^-)ds\\
				&+O\lz N^{1/4} X^{1/2+\epsilon}\pz,
			\end{aligned}
		\end{equation}where
		$T_s(S):=T(\{s_j+s:\ s_j\in S\}).$
	\end{corollary}
	
	All of the corollaries are proved by applying Perron's formula and shifting the resulting integral, while the region in Theorem \ref{thm: main theorem} determines how far the integral can be shifted. To explain why Theorem \ref{thm: main theorem} allows us to detect the 1-swap terms but not higher swaps in various settings, note that to detect the 2-swap terms of size $X^{2-s_{j_1}-s_{j_2}}$, one needs to take $\re(w)<\re(2-s_{j_1}-s_{j_2})$. However, for any quadruple of distinct variables $s_{1},\dots,s_{4}$, the last condition in \eqref{eqn: hyperplanes for A} implies that either $\re(w+s_1+s_2)>2$, or $\re(w+s_3+s_4)>2.$
	
	Similarly, had we not been able to use the functional equation in $w$ (the analogue of Soundararajan's Poisson summation), we would use the region $R_1$ \eqref{eqn:definition of R_1} in place of $R$ in \eqref{eqn: S definition} and the result would contain conditions $\re(s_{j_1}+s_{j_2}+2w)>3$, so we wouldn't be able to detect the 1-swap terms of size $X^{3/2-s_j}$ (or the second and higher moments).
	
	\smallskip
	
	Analogous results with similar proof would also apply to other families of quadratic twists of a fixed $L$-function. The method could also be extended to treat the multiple Dirichlet series associated with mollified moments or the ratios conjectures.
	
	\section*{\large\textbf{Acknowledgements}}
	I would like to thank Brian Conrey, Andrew Granville and Kaisa Matom\"aki for their encouragement and Siegfred Baluyot for several useful discussions on these topics. I am also grateful to the anonymous referee for a careful reading of the manuscript and many interesting comments and suggestions.
	
	The author was supported by Academy of Finland Grant no. 333707 and by the Charles University
	Primus Programmes PRIMUS/24/SCI/010 and PRIMUS/25/SCI/017.

	\section{Preliminaries}
	\subsection{Dirichlet characters and $L$-functions}
	We will work with the family of real Dirichlet characters $\chi_{8d}$ for positive odd square-free integers $d$, given by the Kronecker's symbol $\chi_{8d}=\bfrac{8d}{\cdot}$, but all results can be easily extended to the family of all real primitive characters parametrized by fundamental discriminants. The Kronecker symbol $\chi_{8d}$ is a Dirichlet character modulo $8d$ for any $d\geq 1$.
	We denote by $\psi_1$ the principal character modulo $4$.
	
	For every positive, odd, square-free $d$, $\chi_{8d}$ is a primitive even quadratic Dirichlet character of conductor $8d$, and its $L$-function satisfies the functional equation
	\begin{equation}\label{eqn: fe primitive}
		L(s,\chi_{8d})=\bfrac{8d}{\pi}^{1/2-s}\frac{\Gamma\bfrac{1-s}{2}}{\Gamma\bfrac{s}{2}}L(1-s,\chi_d)=d^{1/2-s}X(s)L(1-s,\chi_d).
	\end{equation}
	
	The generalized Lindel\"of hypothesis is the estimate
	\begin{equation}\label{eqn: GLH for individual L-function}
		|L(s,\chi_{8d})|\ll |d(1+|s|)|^{\max\{0,1/2-\re(s)\}+\epsilon},
	\end{equation} and this bound would follow from a bound for all moments of the form
	\begin{equation}\label{eqn: moments conjecture implies Lindelof}
		\sumstar_{d\leq X}|L(s,\chi_{8d})|^k\ll_k |s|^\epsilon X^{1+\epsilon}
	\end{equation} for each $s\in\C$ with $\re(s)\geq 1/2$ and $k\in\N$. Such a bound is far beyond reach of current technology for large $k$. For $k\leq 4$, we have the following bound due to Heath-Brown \cite[Theorem 2]{HB}, which can serve as a substitute: for any $s$ with $\re(s)\geq 1/2$
	\begin{equation}\label{eqn: Heath-Brown fourth moment}
		\sumstar_{d\leq X}|L(s,\chi_{8d})|^4\ll X^{1+\epsilon}(|s|)^{1+\epsilon}.
	\end{equation}
	
	We will also use the Jacobi symbols $\leg{\cdot}{n}$ for odd $n$. These are Dirichlet characters modulo $n$, which are primitive if and only if $n$ is square-free, and $\leg{\cdot}{n}$ is even for $n\equiv 1\mod 4$ and odd for $n\equiv 3\mod 4$. Writing $n=n_0n_1^2$ with $n_0$ square-free, we have
	\begin{equation}\label{eqn: non-primitive in terms of primitive L-function}
		L\lz s,\leg{\cdot}{n}\pz=L\lz s,\leg{\cdot}{n_0}\pz \prod_{p|n_1}\lz 1-\frac{\leg{p}{n_0}}{p^s}\pz.
	\end{equation} It follows that for $\re(s)>0$, we have the bound
	\begin{equation}\label{eqn: bound non-primitive in terms of primitive L-function}
		\lab L\lz s,\leg{\cdot}{n}\pz\rab \ll n^{\epsilon}\lab L\lz s,\leg{\cdot}{n_0}\pz \rab.
	\end{equation}
	From \eqref{eqn: Heath-Brown fourth moment}, H\"older's inequality, \eqref{eqn: bound non-primitive in terms of primitive L-function} and the functional equation \eqref{eqn: fe primitive}, we obtain
	\begin{equation}\label{eqn: first moment over non-primitive characters}
		\sum_{\substack{n\leq X\\n\odd}} \lab L\lz s,\leg{\cdot}{n}\pz\rab\ll_s X^{\max\{1,3/2-\re(s)\}+\epsilon},\quad \text{for}\quad \re(s)\geq 0,
	\end{equation} where the implied constant depends polynomially on $|s|$.
	
	\subsection{Gauss sums}
	
	Let $\chi$ be a Dirichlet character modulo $q$. The Gauss sum is defined by
	\begin{equation}\label{key}
		\tau(\chi,\l):=\sum_{j\mod q}\chi(j)e\bfrac{j\l}{q}.
	\end{equation}
	A crucial tool for us will be a functional equation valid for all Dirichlet $L$-functions, including those associated with non-primitive characters. The following is \cite[Proposition 2.3]{Cec1} (with a slightly different normalization).
	\begin{proposition}\label{prop: general functional equation}
		Let $\chi$ be any character modulo $n$. Then we have
		\begin{equation}\label{eqn: fe for all L-functions}
			L(s,\chi)=n^{1/2-s}X_{\pm}(s)K(1-s,\chi),
		\end{equation} where
		\begin{equation}\label{key}
			\begin{aligned}
				K(s,\chi)&=\sum_{\l=1}^\infty\frac{\tau(\chi,\l)/\sqrt n}{\l^s},\\
				X_{\pm}(s)&=\begin{cases}
					X_+(s):=\pi^{s-1/2}\frac{\Gamma\bfrac{1-s}{2}}{\Gamma\bfrac s2},&\hbox{if $\chi$ is even,}\\
					X_-(s):=-i\pi^{s-1/2}\frac{\Gamma\bfrac{2-s}{2}}{\Gamma\bfrac {1+s}2},&\hbox{if $\chi$ is odd.}
				\end{cases}
			\end{aligned}
		\end{equation}
	\end{proposition}
	For a primitive character $\chi$, we have
	\begin{equation}\label{key}
		\tau(\chi,\l)=\overline\chi(\l)\tau(\chi,1),
	\end{equation} which can be used in Proposition \ref{prop: general functional equation} to recover the usual functional equation for $L(s,\chi)$.
	
	We remark that $\tau(\chi,\l)$ and $K(s,\chi)$ also implicitly depend on the modulus of $\chi$. They change if $\chi$ is lifted to a character with a larger modulus, but \eqref{eqn: fe for all L-functions} remains true.
	
	The functional equation \eqref{eqn: fe for all L-functions} implies that $K(s,\chi)$ has a meromorphic continuation to $s\in\C$, and it satisfies similar bounds in vertical strips as $L(s,\chi)$. In particular, \eqref{eqn: first moment over non-primitive characters} and \eqref{eqn: fe for all L-functions}, we have
	\begin{equation}\label{eqn: first moment K}
		\sum_{\substack{n\leq X \\ n\odd}} \lab K\lz s,\leg{\cdot}{n}\pz\rab \ll_s X^{\max\{1, 3/2-\re(s)\}+\epsilon}, \quad\text{for} \quad \re(s)\geq 0.
	\end{equation}
	
	We will also work with the modified Gauss sums, defined for the Jacobi symbol $\leg{\cdot}{n}$ by
	\begin{equation}\label{eqn: multiplicative Gauss sums}
		\begin{aligned}
			G\lz\leg{\cdot}{n},\l\pz&=\lz\frac{1-i}{2}+\leg{-1}{n}\frac{1+i}{2}\pz\tau\lz\leg{\cdot}{n},\l\pz\\
			&=\begin{cases}
				\tau\lz\bfrac{\cdot}{n},\l\pz,&\hbox{if $n\equiv 1\mod 4$,}\\
				-i\tau\lz\bfrac{\cdot}{n},\l\pz,&\hbox{if $n\equiv 3\mod 4$,}
			\end{cases}
		\end{aligned}
	\end{equation} which are multiplicative in the $n$-variable: for $(m,n)=1$, we have
	\begin{equation}\label{key}
		G\lz\leg{\cdot}{m},\l\pz G\lz\leg{\cdot}{n},\l\pz=G\lz\leg{\cdot}{mn},\l\pz.
	\end{equation}
	Moreover, if $p$ is an odd prime and $p^a||\l$, we have the following explicit evaluation due to Soundararajan \cite[Lemma 2.3]{Sou}:
	\begin{equation}\label{eqn: Gauss sum evaluation}
		G\lz\leg{\cdot}{p^k},\l\pz=\begin{cases}
			\phi(p^k),&\hbox{if $k\leq a$, $k$ even,}\\		0,&\hbox{if $k\leq a$, $k$ odd,}\\
			-p^a,&\hbox{if $k=a+1$, $k$ even,}\\
			\leg{\l p^{-a}}{p}p^a\sqrt p,&\hbox{if $k=a+1$, $k$ odd,}\\
			0,&\hbox{if $k\geq a+2$.}\\
		\end{cases}
	\end{equation}
	
	\subsection{Multivariable complex analysis}
	A general reference for multivariable complex analysis is \cite{Hor}.

	An open set $T\subset\C^n$ is a tube domain if there is a connected open set $U\subset\R^n$ such that $T=U+i\R^n=\{z\in\C^n:\ \re(z)\in U\}.$ We call $U$ the base set of $T$.
	Tube domains are generalizations of vertical strips and are natural domains of definition of a multiple Dirichlet series.
	
	We denote the convex hull of $T$ by $\ch (T)$.
	The following theorem is immensely useful in obtaining meromorphic continuation of multiple Dirichlet series.
	\begin{theorem}[Bochner's Tube Theorem]
		Any function holomorphic on a tube domain $T$ has a holomorphic extension to the tube domain $\ch(T)$.	
	\end{theorem}

	The following is \cite[Proposition C.5]{Cec1}
	\begin{proposition}\label{prop: extending inequalities}
		Assume that $T\subset \C^n$ is a tube domain, $g,h:T\rightarrow \C$ are holomorphic functions, and let $\tilde g,\tilde h$ be their holomorphic continuation to $\ch(T)$. If  $\lvert g(z)\rvert \leq \lvert h(z)\rvert $ for all $z\in T$, and $h(z)$ is nonzero in $T$, then also $\lvert \tilde g(z)\rvert \leq \lvert \tilde h(z)\rvert $ for all $z\in \ch( T)$.
	\end{proposition}
	
	We will use this proposition to prove that the continuation of the multiple Dirichlet series is polynomially bounded in vertical strips. Indeed, note that if $Z(s_1,\dots,s_k,w)$ is polynomially bounded in vertical strips in the sense of \eqref{eqn: polynomial boundedness def in theorem} in two intersecting tube domains $T_1,T_2$, then Proposition \ref{prop: extending inequalities} implies similar polynomial boundedness in $\ch(T_1\cup T_2)$ (see also \cite[Proposition 2.8]{AlBu})
	
	\subsection{Convex polyhedra} A crucial part of our proof is the determination of the convex hull of a polyhedron in $\R^k$. We now introduce some basic theory that we will use, which can be found in most introductory textbooks in this topic (see for example \cite{Ale}, \cite{Bas}, \cite{HuWe}).
	
	A polyhedron is a set $P\subset\R^n$ of the form
	\begin{equation}\label{eqn: H-representation definition}
		P=\{x\in\R^{n}:Ax\geq b\},
	\end{equation} where $A\in\R^{m\times n}$ is an $m\times n$ real matrix, and $b\in\R^n$. This definition of a polyhedron as intersection of half-spaces is called the $H$-representation of $P$.
	Multiple Dirichlet series are naturally defined on tube domains whose base sets are (open) polyhedra. 
	
	In view of Bochner's Tube Theorem, we will need to find the convex hulls of polyhedra. To do this, a description of a polyhedron in terms of its rays and vertices is more convenient. By the Minkowski-Weyl Theorem, any polyhedron has a $V$-representation in the form
	\begin{equation}\label{eqn: V-representation definition}
		P=\ch(V_1,\dots,V_\l)+\nonneg(r_1,\dots,r_m),
	\end{equation} for some points $V_1,\dots,V_\l\in P$ and vectors $r_1,\dots,r_m\in\R^{n}$, where $\ch(V_1,\dots,V_\l)$ denotes the convex hull of the points $V_1,\dots,V_\l$, \begin{equation}\label{key}
		\nonneg(r_1,\dots,r_m)=\{\lambda_1 r_1+\dots+\lambda_m r_m:\ \lambda_1,\dots,\lambda_m\geq 0\}
	\end{equation} denotes the set of nonnegative combinations of $r_1,\dots,r_m$, and $+$ denotes the Minkowski sum
	$S_1+S_2=\{s_1+s_2:\ s_1\in S_1,\ s_2\in S_2\}.$
	In this representation, the points $V_1,\dots,V_\l$ are called vertices and the vectors $r_1,\dots,r_m$ are called rays. Such a representation is not unique, but it can be obtained by taking the set of extremal points for vertices and the set of extremal rays for rays. The extremal points and extremal rays can be determined from the $H$-representation using the following Proposition:
	\begin{proposition}\label{prop: extremal points and extremal rays definition}
		Let \begin{equation}\label{key}
			P=\{x\in\R^{n}:Ax\geq b\}. 
		\end{equation}  For $x\in P$, let $A_{\mathrm{tight}}(x;b)$ be the set of inequalities from the system $Ax\geq b$ for which there is equality.
		\begin{enumerate}
			\item  A point $x\in P$ is extremal if and only if $A_{\mathrm{tight}}(x;b)$ has rank $n$.
			\item The extremal rays of $P$ are the extremal rays of the polyhedron given by the homogeneous system
			\begin{equation}\label{key}
				P_{\mathrm{hom}}=\{x\in\R^{n}:Ax\geq 0\}.
			\end{equation} A vector $x\in P_{\mathrm{hom}}\setminus\{0\}$ is an extremal ray if and only if $A_{\mathrm{tight}}(x;0)$ has rank $n-1$.
		\end{enumerate}
	\end{proposition}
	We will call the equations formed by the rows of the system $Ax=b$ the boundary hyperplanes of $P$. Proposition \ref{prop: extremal points and extremal rays definition} says that the extremal points lie at the unique intersections of at least $n$ boundary hyperplanes, and the extremal rays are lines obtained by the intersection of at least $n-1$ boundary hyperplanes of the polyhedron given by the homogeneous system $Ax\geq 0$. 
	
	For two polyhedra 
	\begin{equation}\label{key}
		\begin{aligned}
			P_1&=\ch(V_1,\dots,V_{\l_1})+\nonneg(r_1,\dots,r_{m_1}),\\
			P_2&=\ch(W_1,\dots,W_{\l_2})+\nonneg(s_1,\dots,s_{m_2}),
		\end{aligned}
	\end{equation} it is straightforward to verify that their convex hull is
	\begin{equation}\label{eqn: convex hull of V-representations}
		\begin{aligned}
			\ch(P_1,P_2)=&\ch(V_1,\dots,V_{\l_1},W_1,\dots,W_{\l_2})\\
			&+\nonneg(r_1,\dots,r_{m_1},s_1,\dots,s_{m_2}).
		\end{aligned}
	\end{equation}
	Let now $\sigma:\R^n\rightarrow\R^n$ be an affine transformation given by $\sigma(v)=\tilde\sigma(v)+t,$ where $\tilde\sigma$ is linear and $t\in\R^n$. Then if $P$ is a polyhedron given by \eqref{eqn: V-representation definition},
	\begin{equation}\label{key}
		\sigma(P)=\ch(\sigma(V_1),\dots,\sigma(V_\l))+\nonneg(\tilde\sigma (r_1),\dots,\tilde\sigma(r_m)).
	\end{equation} 
	
	\section{Basic properties of $A(s_1,\dots,s_k,w)$}\label{sec: basic properties of A}
	We recall the definition
	\begin{equation}\label{eqn: A definition again}
		A(s_1,\dots,s_k,w)=\sumstar_{d\geq 1}\frac{L(s_1,\chi_{8d})\dots L(s_k,\chi_{8d})}{d^w}.
	\end{equation}
	To find the region of absolute convergence we use Heath-Brown's bound \eqref{eqn: Heath-Brown fourth moment} for $k\leq 4$ or the GLH \eqref{eqn: GLH for individual L-function} for higher $k$ and obtain the region of absolute convergence
	\begin{equation}\label{eqn: basic region 1}
		\{(s_1,\dots,s_k,w)\in\C^{k+1}:\ \re(s_j)\geq 1/2 \text{ for each $j=1,\dots,k$},\ \re(w)>1\}.
	\end{equation}
	From now on, we will shorten our notation and write the set above as \begin{equation}\label{key}
		\{\re(s_j)\geq 1/2,\ \re(w)>1\}.
	\end{equation}
	
	We obtain a second region by expanding each $L$-function into its Dirichlet series and then first evaluating the sum over the family, getting
	\begin{equation}\label{eqn: A after exchanging sums}
		A(s_1,\dots,s_k,w)=\sum_{\substack{n_1,\dots,n_k\\n_1\dots n_k\odd}}\frac{L_D\lz w,\leg{\cdot}{n_1\dots n_k}\pz}{n_1^{s_1}\dots n_k^{s_k}},
	\end{equation} where
	\begin{equation}\label{eqn: L_D definition}
		\begin{aligned}
			L_D\lz w,\leg{\cdot}{n}\pz&=\sumstar_{d\geq 1}\frac{\chi_{8d}(n)}{d^w}=\sum_{d\geq 1}\frac{\psi_1(d) \chi_{8d}(n)\mu^2(d)}{d^w}
			\\
			&=\leg{8}{n}\frac{L\lz w,\leg{\cdot}{n}\psi_1\pz}{L\lz 2w,\leg{\cdot}{n}^2\psi_1\pz}.
		\end{aligned}
	\end{equation}
	Noting that the denominator is $\asymp 1$ whenever $\re(w)>1/2+\delta$ for any $\delta>0$, we see from \eqref{eqn: first moment over non-primitive characters} that \eqref{eqn: A after exchanging sums} is absolutely convergent for
	\begin{equation}\label{eqn: basic region 2}
		\{\re(w)>1/2,\ \re(s_j)>1\},
	\end{equation} up to a simple pole at $w=1$ coming from the terms with $n_1\dots n_k=\square$, with residue
	\begin{equation}\label{key}
		\res{w=1}A(s_1,\dots,s_k,w)=\frac{2}{3\zeta(2)}\sum_{\substack{n_1\dots n_k=\square,\\n_1\dots n_k\odd}}\frac{a(n_1\dots n_k)}{n_1^{s_1}\dots n_k^{s_k}}=T(S),
	\end{equation} where $S$, $T(S)$ and $a(n)$ are as in Conjecture \ref{con: cfkrs}.
	
	By Bochner's Tube Theorem, we have a holomorphic continuation of $(w-1)A(s_1,\dots,s_k,w)$ to the convex hull of \eqref{eqn: basic region 1} and \eqref{eqn: basic region 2}, which is the region
	\begin{equation}\label{eqn:definition of R_1}
		R_1=\{\re(s_j)\geq 1/2,\ \re(w)>1/2,\ \re(s_j+w)>3/2\}.
	\end{equation}
	
	For any $J\subset\{1,\dots,k\}$, applying the functional equation to $L(s_j,\chi_{8d})$ in \eqref{eqn: A definition} for $j\in J$, we obtain the functional equation
	\begin{equation}\label{eqn: A fe}
		A(s_1,\dots,s_k,w)=\prod_{j\in J}X(s_j) \cdot A\lz s_1^J,\dots,s_k^J,w+\sum_{j\in J}s_j-\frac{|J|}{2} \pz,
	\end{equation}
	where
	\begin{equation}\label{key}
		s_j^J=\begin{cases}
			1-s_j,&\hbox{if $j\in J$,}\\
			s_j,&\hbox{ if $j\notin J$.}
		\end{cases}
	\end{equation}
	The pole at $w=1$ gives rise to a new pole at $w=1+\frac{|J|}{2}-\sum\limits_{j\in J}s_j$ with residue
	\begin{equation}\label{eqn: A residues}
		\res{w=1+\frac{|J|}{2}-\sum\limits_{j\in J}s_j}A(s_1,\dots,s_k,w)=\prod_{j\in J}X(s_j)\cdot T(S\setminus S_J\cup S_J^-),
	\end{equation} where $S_J$ and $S_J^-$ are as in Conjecture \ref{con: cfkrs}.
	
	\subsection{Heuristic extra functional equation}\label{sec: heuristic fe in w}
	The obtained region $R_1$ would be enough to compute the 1st moment and detect the diagonal terms in the corollaries. To do better, we use an extra functional equation in the $w$ variable.
	
	Assume that we had \eqref{eqn: A after exchanging sums}, but with $L_D$ replaced by $L$, and that all the characters involved were primitive. We could then apply the functional equation for $L\lz w,\leg{\cdot}{n_1\dots n_k}\pz$, and obtain a relation of the form
	\begin{equation}\label{eqn: heuristic fe in w}
		A(s_1,\dots,s_k,w)\approx A(s_1+w-1/2,\dots,s_k+w-1/2,1-w).
	\end{equation}
	This would give a meromorphic extension to the region in Theorem \ref{thm: main theorem}, and slightly more, since we could iterate the functional equations under $\sigma_J$ and \eqref{eqn: heuristic fe in w}. However, this wouldn't extend beyond the most restricting half-spaces given by $\re(s_{j_1}+s_{j_2}+s_{j_3}+s_{j_4}+2w)>4$.
	
	Obtaining this functional equation rigorously is the main part in the multiple Dirichlet series approach, such as \cite{DGH}, and is achieved by modifying $A(s_1,\dots,s_k,w)$ by inserting some carefully chosen weights into its Dirichlet series. However, it is very hard to find these special weights, and in the end, extra work is required to go back to the original problem.
	
	In the next section, we will obtain an extension of $A(s_1,\dots,s_k,w)$ into a similar region as the one which would come from \eqref{eqn: heuristic fe in w}, without modifying $A(s_1,\dots,s_k,w)$. A crucial tool will be the functional equation \eqref{eqn: fe for all L-functions}.

	\section{The functional equation in $w$ and meromorphic continuation}
	In this section, we provide the meromorphic continuation of 
	\begin{equation}\label{key}
		A(s_1,\dots,s_k,w)=\sumstar_{d\geq 1}\frac{L(s_1,\chi_{8d})\dots L(s_k,\chi_{8d})}{d^w}
	\end{equation} to a similar region as one coming from the heuristic functional equation in $w$ from Section \ref{sec: heuristic fe in w}.
	We have
	\begin{equation}\label{eqn: A in terms of A_c}
		\begin{aligned}
			A(s_1,\dots,s_k,w)&=\sum_{d\geq 1}\mu^2(d)\frac{\psi_1(d)L(s_1,\chi_{8d})\dots L(s_k,\chi_{8d})}{d^w}\\
			&=\sum_{d\geq 1}\frac{\psi_1(d)L(s_1,\chi_{8d})\dots L(s_k,\chi_{8d})}{d^w}\sum_{c^2|d}\mu(c)\\
			&=\sum_{c\geq 1}\frac{\psi_1(c^2)\mu(c)}{c^{2w}}\sum_{d\geq 1}\frac{\psi_1(d)L(s_1,\chi_{8c^2d})\dots L(s_k,\chi_{8c^2d})}{d^w}\\
			&=\sum_{c\geq 1}\frac{\psi_1(c)\mu(c)}{c^{2w}}\cdot A_{(c)}(s_1,\dots,s_k,w),
		\end{aligned}
	\end{equation} where 
	\begin{equation}\label{eqn: A_c first expression}
		A_{(c)}(s_1,\dots,s_k,w)=\sum_{d\geq 1}\frac{\psi_1(d)L(s_1,\chi_{8c^2d})\dots L(s_k,\chi_{8c^2d})}{d^w}.
	\end{equation}
	We also have a second expression obtained by expanding each $L$-function and evaluating the sum over $d$ first:
	\begin{equation}\label{eqn: A_c second expression}
		\begin{aligned}
			A_{(c)}(s_1,\dots,s_k,w)&=\sum_{d,n_1,\dots,n_k\geq 1}\frac{\psi_1(d)\leg{8c^2d}{n_1\dots n_k}}{n_1^{s_1}\dots n_k^{s_k}d^w}\\
			&=\sum_{\substack{d,n_1,\dots,n_k\geq 1\\(n_1\dots n_k,2c)=1}}\frac{\psi_1(d)\leg{2}{n_1\dots n_k}\leg{d}{n_1\dots n_k}}{n_1^{s_1}\dots n_k^{s_k}d^w}\\
			&=\sum_{\substack{n_1,\dots,n_k\geq 1\\ (n_1\dots n_k,2c)=1}}\frac{\psi_2(n_1\dots n_k)L\lz w,\leg{\cdot}{n_1\dots n_k}\cdot \psi_1\pz}{n_1^{s_1}\dots n_k^{s_k}},
		\end{aligned}
	\end{equation}
	where $\psi_2(n)=\bfrac{2}{n}$ is a Dirichlet character modulo 8.
	
	Using \eqref{eqn: non-primitive in terms of primitive L-function} and \eqref{eqn: first moment over non-primitive characters}, assuming that $k\leq 4$ or that GLH holds, \eqref{eqn: A_c first expression} is absolutely convergent in the region
	\begin{equation}\label{key}
		\{\re(w)>1,\ \re(s_j)\geq 1/2\},
	\end{equation}
	and by \eqref{eqn: first moment over non-primitive characters}, the expression \eqref{eqn: A_c second expression} converges absolutely in
	\begin{equation}\label{key}
		\{\re(s_j)>1,\ \re(s_j+w)>3/2, \re(w)>0\},
	\end{equation} except a pole at $w=1$, and it is $\ll c^\epsilon$ in both these regions (away from the pole). By Bochner's Tube Theorem and Proposition \ref{prop: extending inequalities}, $A_{(c)}(s_1,\dots,s_k,w)$ has a meromorphic continuation to
	\begin{equation}\label{convex hull}
		\{\re(s_j)\geq 1/2,\ \re(w)>0,\ \re(s_j+w)>3/2\},
	\end{equation} and it is $\ll c^\epsilon$ here.
	
	Now we use the functional equation \eqref{eqn: fe for all L-functions} in \eqref{eqn: A_c second expression}
	and obtain
	\begin{equation}\label{eqn: fe for A_c}
		\begin{aligned}
			A_{(c)}&(s_1,\dots,s_k,w)\\
			&=X_+(w)\sum_{\substack{(n_1\dots n_k,2c)=1,\\n_1\dots n_k\equiv 1\mod 4}}\frac{\psi_2(n_1\dots n_k)K\lz 1-w,\leg{\cdot}{n_1\dots n_k}\pz}{n_1^{s_1+w-1/2}\dots n_k^{s_k+w-1/2}}\\
			&+X_-(w)\sum_{\substack{(n_1\dots n_k,2c)=1,\\n_1\dots n_k\equiv 3\mod 4}}\frac{\psi_2(n_1\dots n_k)K\lz 1-w,\leg{\cdot}{n_1\dots n_k}\pz }{n_1^{s_1+w-1/2}\dots n_k^{s_k+w-1/2}}.
		\end{aligned}
	\end{equation}
	\begin{sloppypar}
		Both sums on the right-hand side above can be written as a combination of the terms $B_{(c)}(s_1+w-1/2,\dots,s_k+w-1/2,1-w;\psi)$ times some gamma factors, where $\psi$ is a character modulo $8$, and
	\end{sloppypar}
	\begin{equation}\label{key}
		B_{(c)}(s_1,\dots,s_k,w;\psi)=\sum_{\substack{(n_1,\dots,n_k,2c)=1}}\frac{\psi(n_1\dots n_k)K\lz w,\leg{\cdot}{n_1\dots n_k}\pz }{n_1^{s_1}\dots n_k^{s_k}}
	\end{equation} (see also \cite[Section 6.4]{Cec1} for a detailed computation).

	Setting $m=n_1\dots n_k$, we have $B_{(c)}(s_1,\dots,s_k,w;\psi)=B_{s_1-s,\dots,s_k-s}(s,w;\psi)$, where
	\begin{equation}\label{eqn: B_t}
		B_{t_1\dots,t_k}(s,w,\psi)=\sum_{(m,2c)=1}\frac{\psi(m)K\lz w,\leg{\cdot}{m}\pz}{m^s}\sum_{m=n_1\dots n_k}\frac{1}{n_1^{t_1}\dots n_k^{t_k}}.
	\end{equation}
	With the change of variables $t_j=s_j-s$, we think of $s$ as the minimum of the real parts of $s_j$, so that $\re(t_j)\geq 0$ for every $j=1,\dots,k.$
	
	By \eqref{eqn: first moment K}, the expression \eqref{eqn: B_t} converges absolutely for $\re(t_j)\geq 0$ in the region
	\begin{equation}\label{eqn: B_t first region}
		\{\re(s)>1, \re(w)>1/2\}.
	\end{equation}

	Our next goal is to prove that for $\re(t_j)\geq 0$, $B_{t_1,\dots,t_k}(s,w;\psi)$ has a meromorphic continuation to the region
	\begin{equation}\label{key}
		\{\re(s)>1/2,\ \re(w)>1\},
	\end{equation} with possible poles only at $s+t_j=1$, and it is $\ll c^\epsilon$ in this region away from the poles.
	
	To prove this, we first expand $B_{t_1,\dots,t_k}(s,w;\psi)$ into a double sum, replace the Gauss sums by their multiplicative analogue, and then first compute the sum over $m$, which can be analytically continued because it will be an Euler product.
	
	We define 
	\begin{equation}\label{key}
		f_{t_1,\dots,t_k}(m)=\sum_{m=n_1\dots n_k}\frac{1}{n_1^{t_1}\dots n_k^{t_k}},
	\end{equation}
	so that
	\begin{equation}\label{key}
		B_{t_1,\dots,t_k}(s,w;\psi)=\sum_{\substack{m,n\geq 1\\(m,2c)=1}}\frac{\psi(m)\tau\lz\leg{\cdot}{m},n\pz f_{t_1,\dots,t_k}(m)}{m^{s+1/2}n^w},
	\end{equation} and using \eqref{eqn: multiplicative Gauss sums}, the above can be written as a combination of the expressions
	\begin{equation}\label{key}
		\tilde B_{t_1,\dots,t_k}(s,w;\psi)=\sum_{n\geq 1}\frac{1}{n^w}\sum_{(m,2c)=1}\frac{\psi(m) G\lz\leg{\cdot}{m},n\pz f_{t_1,\dots,t_k}(m)}{m^{s+1/2}},
	\end{equation} with constant coefficients, where $\psi$ ranges over characters modulo 8.
	Let us now denote the inner sum as $D_{t_1,\dots,t_k}(s;\psi,n,c)$, so
	\begin{equation}\label{key}
		D_{t_1,\dots,t_k}(s;\psi,n,c)=\sum_{(m,2c)=1}\frac{\psi(m) G\lz\leg{\cdot}{m},n\pz f_{t_1,\dots,t_k}(m)}{m^{s+1/2}}.
	\end{equation}
	\begin{lemma}\label{lemma: continuation of D}
		We have
		\begin{equation}\label{key}
			D_{t_1,\dots,t_k}(s;\psi,n,c)=\prod_{j=1}^k\frac{L\lz s+t_j,\leg{4n}{\cdot}\psi\pz}{L\lz 2s+2t_j,\leg{4n}{\cdot}^2\psi^2\pz}\cdot Z_{c,n}(t_1,\dots,t_k,s),
		\end{equation} where $Z_{c,n}(t_1,\dots,t_k,s)$ is an Euler product that is absolutely convergent for $\re(t_j)>0$ and $\re(s)>1/2$, and is $\ll |cn|^\epsilon$ in this region.
	\end{lemma}
	\begin{proof}
		Since all the coefficients are multiplicative in $m$, we have the Euler product
		\begin{equation}\label{key}
			D_{t_1,\dots,t_k}(s;\psi,n,c)=\prod_{p\nmid 2c}\lz\sum_{j\geq 0}\frac{\psi(p^j)G\lz\leg{\cdot}{p^j},n\pz f_{t_1,\dots,t_k}(p^j)}{p^{j(s+1/2)}}\pz=P_1(s)P_2(s),
		\end{equation} where $P_1$ is the product over $p\nmid 2cn$, and $P_2$ is the rest. 
		
		We first bound $P_2(s)$. Note that $P_2(s)$ is a finite product and \eqref{eqn: Gauss sum evaluation} implies that each factor is a sum of finitely many terms, so it is a holomorphic function of $s,t_1,\dots,t_k$. Let $n=\prod_jp_j^{a_j}$ be the prime factorization of $n$ and denote by $n'$ the product of the prime powers $p_j^{a_j}$ with $p_j\nmid 2c$. From \eqref{eqn: Gauss sum evaluation}, we see that $\lab G\lz\leg{\cdot}{m},n\pz\rab\leq m$ and we also have $\lab f_{t_1,\dots,t_k}(m)\rab\leq \tau_k(m)$ (the $k$-fold divisor function) for $\re(t_j)\geq 0$, therefore
		\begin{equation}\label{key}
			|P_2(s)|\leq\prod_{p^a || n'} \lz1+\sum_{j=1}^{a+1}\frac{p^j\tau_k(p^j)}{p^{j(\re(s)+1/2)}}\pz\leq \prod_{p^a|| n'}\lz1+\frac{\tau_k(p^{a+1})}{p^{\re(s)-1/2}-1}\pz.
		\end{equation} For $\re(s)>1/2+\delta$, let $C(\delta)>\max\{2,\frac{1}{2^{\delta}-1}\}$. Then
		\begin{equation}\label{key}
			\begin{aligned}
				|P_2(s)|&\leq\prod_{p^a|| n'}\lz 1+\tau_k(p^{a+1}) C(\delta)\pz\leq \prod_{p^a||n'}\lz\tau_k(p^{2a})C(\delta)\pz\\
				&\leq \tau_k(n'^2)C(\delta)^{\omega(n')}\ll n^\epsilon.
			\end{aligned}
		\end{equation}
		
		Now we deal with $P_1(s)$. For an odd prime $p\nmid n$, \eqref{eqn: Gauss sum evaluation} implies that \begin{equation}\label{key}
			G\lz\leg{\cdot}{p^j},n\pz=\begin{cases}
				1,&\hbox{if $j=0$,}\\
				\leg{n}{p}\sqrt p,&\hbox{if $j=1$},\\
				0,&\hbox{if $j\geq 2$,}
			\end{cases},
		\end{equation} so
		\begin{equation}\label{key}
			\begin{aligned}
				P_1(s)&=\prod_{p\nmid 2cn}\lz1+\frac{\psi(p)\leg{n}{p}f_{t_1,\dots,t_k}(p)}{p^s}\pz\\
				&=\prod_{p\nmid c}\lz1+\frac{\psi(p)\leg{4n}{p}}{p^s}\lz\frac{1}{p^{t_1}}+\dots+\frac{1}{p^{t_k}}\pz\pz.
			\end{aligned}
		\end{equation}
		Now using that for $\re(s)>0$,
		\begin{equation}\label{key}
			\begin{aligned}
				&\lz1+\frac{\chi(p)}{p^s}\lz p^{-t_1}+\dots +p^{-t_k}\pz\pz\lz1+\frac{\chi(p)}{p^{s+t_1}}\pz^{-1}\\ 
				&=1+\frac{\chi(p)}{p^s}\lz p^{-t_2}+\dots +p^{-t_k}\pz\lz1+\frac{\chi(p)}{p^{s+t_1}}\pz^{-1}\\
				&=1+\frac{\chi(p)}{p^s}\lz p^{-t_2}+\dots +p^{-t_k}\pz\lz1-\frac{\chi(p)}{p^{s+t_1}+\chi(p)}\pz\\
				&=1+\frac{\chi(p)}{p^s}\lz p^{-t_2}+\dots +p^{-t_k}\pz+O\lz p^{-2s-t_1}\pz,
			\end{aligned}
		\end{equation} together with
		\begin{equation}\label{key}
			1+\frac{\chi(p)}{p^{s+t_1}}=\lz 1-\frac{\chi(p)^2}{p^{2s+2t_1}}\pz\lz1-\frac{\chi(p)}{p^{s+t_1}}\pz^{-1},
		\end{equation} we inductively get
		\begin{equation}\label{key}
			\begin{aligned}
				P_1(s)&=\prod_{j=1}^k\frac{L_{(c)}\lz s+t_j,\psi\leg{4n}{\cdot}\pz}{L_{(c)}\lz 2s+2t_j, \psi^2\leg{4n}{\cdot}^2\pz}\cdot E(t_1,\dots,t_k,s)\\
				&=\prod_{j=1}^k\lz\frac{L\lz s+t_j,\psi\leg{4n}{\cdot}\pz}{L\lz 2s+2t_j,\psi^2\leg{4n}{\cdot}^2\pz}\cdot\prod_{p|c}\lz 1+\frac{\leg{4n}{p}\psi(p)}{p^{s+t_j}}\pz^{-1}\pz E(t_1,\dots,t_k,s),
			\end{aligned}
		\end{equation} where
		\begin{equation}\label{key}
			E(t_1,\dots,t_k,s)=\prod_{p\nmid c}\lz1+O\lz p^{-2s-t_1}+\dots +p^{-2s-t_k}\pz\pz.
		\end{equation}
		The lemma follows after setting \begin{equation}\label{key}
			Z_{c,n}(t_1,\dots,t_k,s)= P_2(s) E(t_1,\dots,t_k,s)\prod_{j=1}^k\prod_{p|c}\lz1+\frac{\leg{4n}{p}\psi(p)}{p^{s+t_j}}\pz^{-1},
		\end{equation}
		and the bound
		\begin{equation}\label{key}
			\prod_{p|c}\lz1+\frac{\leg{4n}{p}\psi(p)}{p^{s+t_j}}\pz^{-1}\leq\tau(c).
		\end{equation}
	\end{proof}
	
	Let us now assume that either $k\leq 4$, or that GLH holds. Then Lemma \ref{lemma: continuation of D} implies that for $\re(t_j)\geq 0$, $\tilde B_{t_1,\dots,t_k}(s,w;\psi)$, and hence also $ B_{t_1,\dots,t_k}(s,w;\psi)$ has a meromorphic continuation to the region \begin{equation}\label{key}
		\{\re(s)>1/2, \re(w)>1\},
	\end{equation} with only poles at $s+t_j=1$ which appear from the terms with $n=\square$ if $\psi$ is a principal character, and it is $\ll |c|^\epsilon$ in this region away from the poles. This region together with the region from \eqref{eqn: B_t first region} implies that for $\re(t_j)\geq 0$, $B_{t_1,\dots,t_k}(s,w;\psi)$ has a meromorphic continuation to
	\begin{equation}\label{eqn: region B}
		\{\re(s)>1/2,\ \re(w)>1/2,\ \re(s+w)>3/2\}.
	\end{equation}
	
	Upon setting $s=\min\limits_{1\leq j\leq k}\{\re(s_j)\}$, it follows that $B_{(c)}(s_1,\dots,s_k,w;\psi)$ has a meromorphic continuation to the region
	\begin{equation}\label{key}
		\{\re(s_j)>1/2, \re(w)>1/2,\ \re(s_j+w)>3/2\},
	\end{equation} with possible poles at $s_j=1$, and it is $\ll c^\epsilon$ away from the poles.
	
	The functional equation \eqref{eqn: fe for A_c}
	provides a continuation of $A_{(c)}(s_1,\dots,s_k,w)$ to the region 
	\begin{equation}\label{eqn: extra region for A_c}
		\{\re(s_j+w)>1,\ \re(w)<1/2, \re(s_j)>1\},
	\end{equation}
	with poles at $s_j+w=3/2$ ($X_{\pm}(w)$ are holomorphic in this region). Note that the poles are some of those already found at the end of Section \ref{sec: basic properties of A}, namely those corresponding to the $1$-swap terms.
	
	\begin{remark}\label{remark: only use fe in w for continuation}
		The fact that we use Lemma \ref{lemma: continuation of D} only to obtain the analytic continuation and not to compute the residues is the reason why our computations of the main terms are simpler than in other works on moments in this family such as \cite{CoRo} or \cite{Sou}, and that they appear exactly in the form predicted by the recipe. Note that the application of the functional equation in \eqref{eqn: fe for A_c} corresponds to the application of the Poisson summation formula in \cite{CoRo} or \cite{Sou}, and a direct computation of the residues in Lemma \ref{lemma: continuation of D} would lead to a similar expression for the 1-swap terms as that obtained in the cited papers. It is thus not surprising that one needs to apply the functional equation in these situations to recover the main terms as they appear in Conjecture \ref{con: cfkrs}.
	\end{remark}
	
	By Bochner's Tube Theorem, $A_{(c)}(s_1,\dots,s_k,w)$ has a meromorphic continuation to
	the convex hull of \eqref{eqn: extra region for A_c} and \eqref{convex hull}, which is the region
	\begin{equation}\label{key}
		\{\re(s_j)\geq 1/2,\ \re(s_j+w)>1,\ \re(2s_j+w)>2\}.
	\end{equation} 
	By Proposition \ref{prop: extending inequalities}, it is $\ll |c|^\epsilon$ in this region away from the poles.
	
	It follows that
	\begin{equation}\label{key}
		A(s_1,\dots,s_k,w)=\sum_{c\geq 1}\frac{\psi_1(c)\mu(c)}{c^{2w}}A_{(c)}(s_1,\dots,s_k,w)
	\end{equation} has a meromorphic continuation to the region
	\begin{equation}\label{eqn: region R}
		R=\{\re(s_j)\geq 1/2,\ \re(w)>1/2,\ \re(2s_j+w)>2\}.
	\end{equation}
	Note that $R$ contains the previously found region $R_1$ from \eqref{eqn:definition of R_1}. It will be useful in the next section to find the $V$-representation of the set $U$, the closure of the base set of the tube domain $R$.
	\begin{proposition}\label{prop: V-representation of R}
		Let
		\begin{equation}\label{key}
			U=\{(x_1,\dots,x_k,z): x_j\geq 1/2,\ z\geq1/2,\ 2x_j+z\geq2\}.
		\end{equation}
		Then the $V$-representation of $U$ is
		\begin{equation}\label{key}
			U=\ch(P,Q)+\nonneg(v_1,\dots,v_{k+1}),
		\end{equation} where $P=(3/4,\dots,3/4,1/2)$, $Q=(1/2,\dots,1/2,1)$, and $v_j=(\underbrace{0,\dots,0}_{j-1},1,\allowbreak0,\dots,0)$ is the $j$th standard basis vector.
	\end{proposition}
	\begin{proof}
		We will apply Proposition \ref{prop: extremal points and extremal rays definition} to find the extremal points and rays of $U$.
		
		First we find the extremal points, which lie at the intersection of at least $k+1$ defining hyperplanes of $U$. 
		
		Note that if $(x_1,\dots,x_k,z)\in U$, and $z=1/2$, we cannot have $x_j=1/2$ for any $j$, as then $2x_j+z=3/2<2.$
		Hence if $E=(x_1,\dots,x_k,z)$ is an extremal point with $z=1/2$, then the only other hyperplanes it can lie on are $2x_j+z=2$, which gives $E=P$.
		
		Otherwise, $z>1/2$, and all the remaining hyperplanes $x_j=1/2$ and $2x_j+z=2$ intersect at $Q$, so this is the only other extremal point.
		
		Now we find the rays of $U$, which lie at the intersection of at least $k$ different hyperplanes of the homogeneous system
		\begin{equation}\label{key}
			\begin{aligned}
				x_j&\geq 0,\\
				z&\geq 0,\\
				2x_j+z&\geq 0.
			\end{aligned}
		\end{equation}
		If $r=(x_1,\dots,x_k,z)$ is a ray with $z=0$, then for each $j$, $x_j=0$ or $2x_j+z=0$ both imply $x_j=0$. Therefore the only rays up to normalization with $z=0$ are $v_1,\dots,v_k$.
		
		If $r$ is a ray with $z>0$, then it can't lie on any hyperplane $2x_j+z=0$ as otherwise $x_j<0$. Therefore the only ray with $z>0$ normalized to $z=1$ is $v_{k+1}$.
	\end{proof}

	\section{Determining the region of meromorphic continuation}\label{sec: convex hull}
	We found that $A(s_1,\dots,s_k,w)$ is defined in the region $R$. It also satisfies the functional equations \eqref{eqn: A fe} under the transformations $\sigma_J$ for $J\subset\{1,\dots,k\}$, where
	\begin{equation}\label{key}
		\sigma_J:(s_1,\dots,s_k,w)\mapsto\lz s_1^J,\dots,s_k^J,w+\sum_{j\in J}s_j-\frac{|J|}{2}\pz.
	\end{equation} Therefore using Bochner's tube theorem, we conclude that it has a meromorphic continuation to the region
	\begin{equation}\label{eqn: S definition}
		S=\ch\lz\bigcup\limits_{J\subset\{1,\dots,k\}}\sigma_J(R)\pz.
	\end{equation}
	Note that there are a few technical issues, namely that $R$ is not open and that the sets $\sigma_J(R)$ intersect for different $J\subset\{1,\dots,k\}$ only at the boundary, but both of these can be easily resolved by allowing $\re(s_j)<1/2$ in \eqref{eqn: basic region 1}. We discuss these further in Section \ref{sec: proof of main theorem}.
	
	\smallskip
	
	In this section, we find an $H$-representation of the region $S$.
	\begin{theorem}\label{thm: region for main theorem}
		The region $S$ is the intersection of the half-spaces given by the inequalities
		\begin{equation}\label{hyperplanes}
			\begin{aligned}
				\re(w)&>1/2,\\
				\re(s_j+2w)&>7/4,\hspace{20pt} j\in\{1,\dots,k\},\\
				\re(s_{j_1}+s_{j_2}+2w)&>5/2,\hspace{20pt} \hbox{$j_1,j_2\in\{1,\dots,k\}$,}\\
				\re(s_{j_1}+s_{j_2}+s_{j_3}+2w)&>13/4,\hspace{14pt} \hbox{$j_1,j_2,j_3\in\{1,\dots,k\}$,}\\
				\re(s_{j_1}+s_{j_2}+s_{j_3}+s_{j_4}+2w)&>4,\hspace{32pt} \hbox{$j_1,j_2,j_3,j_4\in\{1,\dots,k\}$, }
			\end{aligned}
		\end{equation} (with $j_1,j_2,j_3,j_4$ pairwise different) and their reflections under the transformations $\sigma_J$, $J\subset\{1,\dots,k\}$.
	\end{theorem}
	
	Let $V$ be the region defined in Theorem \ref{thm: region for main theorem}.
	
	The inclusion $S\subset V$ holds because $V$ is convex, contains $R$ and is invariant under each $\sigma_J$, $J\subset\{1,\dots,k\}$. To see that $R\subset V$, note that if $(s_1,\dots,s_k,w)\in R$, then $\re(w)>1/2$ and $\re(2s_j+w)>2$ imply each of the following:
	\begin{equation}\label{key}
		\begin{aligned}
			\re(2s_{j_1}+4w)&>2+\frac32=\frac72,\\
			\re(2s_{j_1}+2s_{j_2}+4w)&>4+1=5,\\
			\re(2s_{j_1}+2s_{j_2}+2s_{j_3}+4w)&>6+1/2=\frac{13}{2},\\
			\re(2s_{j_1}+2s_{j_2}+2s_{j_3}+2s_{j_4}+4w)&>8.
		\end{aligned}
	\end{equation} This implies that every point in $R$ satisfies the inequalities stated in \eqref{hyperplanes}. To see that they also satisfy the conditions coming from the reflections under $\sigma_J$, we refer to the description from Remark \ref{remark: reflections of hyperplanes} (i), where the claim now follows from the fact that $\re(s_j)\geq 1/2$ for $(s_1,\dots,s_k,w)\in R$.
	
	To prove that $V\subset S$, we find the $V$-representation of $T$, the closure of the base set of the tube domain $V$.
	
	An inequality in the definition of $T$ has either the form (recall Remark \ref{remark: reflections of hyperplanes}; we write $x_j$ for $\re(s_j)$ and $z$ for $\re(w)$)
	\begin{equation}\label{key}
		x_{j_1}+\dots+x_{j_\l}+z\geq\frac{1+\l}{2},
	\end{equation} for $\l=0,\dots,k,$ which we call a type 0, level $\l$ inequality, or
	\begin{equation}\label{key}
		x_{i_1}+\dots +x_{i_t}+2x_{j_1}+\dots+2x_{j_\l}+2z\geq1+\frac{3t}{4}+\l,
	\end{equation} for $t=1,\dots,4$ and $\l=0,\dots,k-t$, which we call a type $t$, level $\l$ inequality. The hyperplanes given by the corresponding equations are called type $t$, level $\l$ hyperplanes. We also call the hyperplanes of type $t\geq 1$ and level $0$ basic hyperplanes.
	
	First, we find the extremal points of $T$.
	\begin{proposition}\label{prop: extremal points}
		The extremal points of $T$ are $\sigma_J(P)$ for $J\subset\{1,\dots,k\}$, where $P=(3/4,\dots,3/4,1/2)$, and $Q=(1/2,\dots,1/2,1)$ if $k\geq 5$.
	\end{proposition}
	We remark that in our problem of computing the moments of $L-functions$ at the central point, the role of $Q$ is especially important for us -- in particular, it lies in the interior of $T$ if $k\leq 3$, allowing us to compute the corresponding moments, it lies along a boundary hyperplane for $k=4$, and becomes a vertex of $T$ if $k=5$.
	
	To prove Proposition \ref{prop: extremal points}, We first establish some lemmas.
	\begin{lemma}\label{lemma: type 0 hyperplanes}
		\begin{enumerate}
			\item Assume that $E\in T$. Then $E$ lies on at most one type 0 hyperplane.
			\item Assume that $E$ is an extremal point of $T$ that lies on a type 0 hyperplane. Then $E=\sigma_J(3/4,\dots,3/4,1/2)$ for some $J\subset \{1,\dots,k\}$.
		\end{enumerate}
	\end{lemma}
	\begin{proof}
		Assume that $E=(x_1,\dots,x_k,z)$ lies on some type 0, level $\l$ hyperplane 
		\begin{equation}\label{key}
			x_{j_1}+\dots+x_{j_\l}+z=\frac{1+\l}{2}.
		\end{equation} We may assume that $\l=0$, as otherwise we may replace $E$ by $\sigma_J(E)$, where $J=\{j_1,\dots,j_\l\}.$  Therefore we have $z=1/2$. The type $1$ inequalities now imply that $x_j\geq 3/4$ for each $j=1,\dots,k$, so $E$ can't lie on any other type $0$ hyperplane. This proves (i).

		To prove (ii), assume that $E=(x_1,\dots,x_k,z)$ is extremal and lies on the type $0$ hyperplane
		\begin{equation}
			x_{j_1}+\dots+x_{j_\l}+z=\frac{1+\l}{2}.
		\end{equation} Then $\sigma_{\{j_1,\dots,j_\l\}}(E)=:\widetilde E=:(\widetilde x_1,\dots,\widetilde x_k,\widetilde z)$ satisfies $\widetilde z=1/2$, so, as above, it also holds that $\widetilde x_j\geq 3/4$ for each $j=1,\dots,k$. It follows that $\widetilde E$ does not lie on any hyperplane of level $\geq 1$: this follows for type 0 from part (i), and for type $1\leq t\leq 4$, $1\leq \l\leq k-t$, it follows from the inequality \begin{equation}\label{key}
			x_{i_1}+\dots+x_{i_t}+2x_{j_1}+\dots+2x_{j_\l}+2z\geq \frac{3t}{4}+2\cdot\frac{3\l}{4}+1>1+\frac{3t}{4}+\l.
		\end{equation}
		Therefore $\widetilde E$ lies only on level 0 hyperplanes, and since they all intersect at $(3/4,\dots,3/4,1/2)$, we must have $\widetilde E=(3/4,\dots,3/4,1/2)$. Then $E=\sigma_{\{j_1,\dots,j_\l\}}(3/4,\dots,3/4,1/2)$.
	\end{proof}
	
	\begin{lemma}\label{Main lemma extremal points}
		Assume that a point $A\in T$ lies on a type $t$, level $\l$ hyperplane $H$ with $t,\l\geq 1$. Assume that $H=\sigma_J(B)$, where $B$ is a basic hyperplane. 
		
		Then $A$ can lie on a basic hyperplane $B'$ only if
		\begin{enumerate}
			\item $B'$ is invariant under $\sigma_J$, or
			\item $B'$ and $H$ are of type 4. In this case, if $\sigma_J(B')\neq B'$, $A$ doesn't lie on any basic hyperplane of type $\leq 3.$
		\end{enumerate}
	\end{lemma}
	\begin{proof}
		Let $A=(x_1,\dots,x_k,z)$ and $J=\{j_1,\dots,j_\l\}$. Then since $A\in H=\sigma_J(B)$, it satisfies an equation of the form
		\begin{equation}\label{eq}
			H:\hspace{20pt}x_{i_1}+\dots+x_{i_t}+2x_{j_1}+\dots+2x_{j_\l}+2z=\frac{4+3t}{4}+\l.
		\end{equation}
		
		Let us now take any $1\leq n\leq \l$. Then the inequality given by the hyperplane $\sigma_{J\setminus\{j_n\}}(B)$ gives
		\begin{equation}\label{}
			x_{i_1}+\dots+x_{i_t}+2x_{j_1}+\dots+2x_{j_\l}-2x_{j_n}+2z\geq \frac{4+3t}{4}+\l-1,
		\end{equation} so we have
		\begin{equation}\label{}
			\frac{4+3t}{4}+\l=x_{i_1}+\dots+x_{i_t}+2x_{j_1}+\dots+2x_{j_\l}+2z\geq 2x_{j_n}+\frac{4+3t}{4}+\l-1,
		\end{equation}
		so \begin{equation}\label{eqn: x_j_n<1/2}
			x_{j_n}\leq 1/2.
		\end{equation}
		
		Now assume that $A\in B'$ for some basic hyperplane $B'$ of type $u$. First, we assume that $u\neq 4$ and show that $B'$ must be invariant under $\sigma_J$.
		
		Assume for contradiction that $B'$ is not invariant under $\sigma_J$ and has type $u\leq 3$. Then we have
		\begin{equation}\label{eqn: basic equality for u<4}
			x_{k_1}+\dots+x_{k_u}+2z=\frac{4+3u}{4}
		\end{equation} for some $k_1,\dots,k_u$, and since it is not invariant under $\sigma_J$, there is some $1\leq n\leq\l$ such that
		$j_n\notin\{k_1,\dots,k_u\}$. Since $A\in T$ and $u<4$, we may consider the basic inequality of type $u+1$
		\begin{equation}\label{key}
			x_{k_1}+\dots+x_{k_u}+x_{j_n}+2z\geq \frac{4+3(u+1)}{4},
		\end{equation} so \eqref{eqn: x_j_n<1/2} gives
		\begin{equation}\label{key}
			x_{k_1}+\dots+x_{k_u}+2z\geq \frac{4+3(u+1)}{4}-x_{j_n}\geq \frac{4+3u}{4}+\frac14,
		\end{equation} which is a contradiction with \eqref{eqn: basic equality for u<4}.
		
		Now we prove that if $A\in B'$, and $B'$ is not invariant under $\sigma_J$, then $t=4.$ We know from the previous part that $u=4$ in this case, so $B'$ is given by the equation
		\begin{equation}\label{key}
			B':\hspace{20pt}	x_{k_1}+\dots+x_{k_4}+2z=4.
		\end{equation} 
		For every $j\notin\{k_1,\dots,k_4\}$, we have
		\begin{equation}\label{key}
			5\leq x_{k_1}+\dots+x_{k_4}+2x_j+2z=4+2x_j,
		\end{equation} so $x_j\geq 1/2$. Recall from \eqref{eqn: x_j_n<1/2} that $x_{j_1},\dots,x_{j_\l}\leq 1/2.$ Since $B'$ is not invariant under $\sigma_J$, there is some $1\leq n\leq \l$ such that $j_n\notin\{k_1,\dots,k_4\}$, and it follows that \begin{equation}\label{eqn: x_j_n=1/2}
			x_{j_n}=1/2.
		\end{equation} Now assume for contradiction that $t\leq 3$, and consider the type $t+1$, level $\l-1$ expression
		\begin{equation}\label{key}
			x_{i_1}+\dots+x_{i_t}+x_{j_n}+2x_{j_1}+\dots+2x_{j_\l}-2x_{j_n}+2z.
		\end{equation} Since $A\in T$, this must be $\geq \frac{4+3(t+1)}{4}+\l-1=\frac{4+3t}{4}+\l-\frac14$, but from $A\in H$ and $x_{j_n}=1/2$, this equals $\frac{4+3t}{4}+\l-\frac12,$ which is a contradiction. Therefore $t=4$.
		
		It remains to prove that in this case, $A$ doesn't lie on any basic hyperplane of type $\leq 3$. Assume that such a hyperplane exists and is given by the equation
		\begin{equation}\label{key}
			x_{m_1}+\dots+x_{m_v}+2z=\frac{4+3v}{4}.
		\end{equation} Since $v\leq 3$, for any $j\notin\{m_1,\dots,m_v\}$, we have
		\begin{equation}\label{key}
			\frac{4+3(v+1)}{4}\leq x_{m_1}+\dots +x_{m_v}+x_j+2z=\frac{4+3v}{4}+x_j,
		\end{equation} so $x_j\geq 3/4.$ Let $w$ be such that $k_w\notin\{m_1,\dots,m_v\}$, without loss of generality $w=4$.
		Then $x_{k_4}\geq 3/4$, but also from
		\begin{equation}\label{eqn: two conditions}
			x_{k_1}+x_{k_2}+x_{k_3}+x_{k_4}+2z=4, \hbox{ and } x_{k_1}+x_{k_2}+x_{k_3}+2z\geq 13/4,
		\end{equation} we find that $x_{k_4}\leq 3/4$, so $x_{k_4}=3/4$, and we have equality in the inequality from \eqref{eqn: two conditions}. But then for $x_{j_n}$ from \eqref{eqn: x_j_n=1/2},
		\begin{equation}\label{key}
			x_{k_1}+x_{k_2}+x_{k_3}+x_{j_n}+2z=\frac{13}{4}+\frac12<2,
		\end{equation} which is a contradiction with $A\in T$.
	\end{proof}
	
	\begin{proof}[Proof of Proposition \ref{prop: extremal points}]
		The extremal points of $T$ are at intersections of at least $k+1$ of the boundary hyperplanes.
		
		The point $P=(3/4,\dots,3/4,1/2)$ is the unique intersection of all basic hyperplanes, so is an extremal point. Also the point $\sigma_J(P)$ is the unique intersection of all level $|J|$ hyperplanes.
		
		The point $Q=(1/2,\dots,1/2,1)$ lies on the type 4 hyperplanes of any level. There is none such hyperplane for $k\leq 3$ and only one for $k=4$. For $k\geq 5$, there are more than $k+1$ of them and $Q$ is their unique intersection, so it is extremal in this case.
		
		Therefore the given points are extremal. It remains to show that there are no others.
		
		Assume that $E$ is an extremal point of $T$ that lies on at least $k+1$ hyperplanes. If $E$ lies on a hyperplane of type 0, it is of the form $\sigma_J(P)$ by Lemma \ref{lemma: type 0 hyperplanes}. If $E$ lies only on hyperplanes of type $4$, then $E=Q$. 
		
		Otherwise, $E$ lies on some hyperplane $B'$ of type $t$ with $1\leq t\leq 3$. We may assume that $B'$ is basic, as otherwise we can replace $E$ by $\sigma_J(E)$ for some $J\subset\{1,\dots,k\}$.
		
		Assume now that $E\in \bigcap\limits_{H\in\mathcal H}H$, where $\mathcal H$ is the set of all hyperplanes that contain $E$. We show that upon possibly replacing $E$ by $\sigma_{J'}(E)$ for some ${J'}\subset\{1,\dots,k\}$, we may assume that all hyperplanes in $\mathcal H$ are basic. Let $\mathcal H=\mathcal H_{B}\cup \mathcal H_{NB}$, where $\mathcal H_B$ consists of all basic hyperplanes that contain $E$ and $\mathcal H_{NB}$ of the rest. We already know that $\mathcal H_B$ is nonempty and contains some hyperplane of type $1,2$ or $3$. For any $H\in \mathcal H_{NB}$, assume that $H=\sigma_{J_1}(B)$ where $B$ is basic. Then by Lemma \ref{Main lemma extremal points}, $\sigma_{J_1}(B')=B'$ for each $B'\in\mathcal H_B$, so we may replace $E$ by $\sigma_{J_1}(E)$, the non-basic hyperplane $H$ by the basic hyperplane $B$ while ensuring that all the elements of $\mathcal H_B$ remain basic.
		
		Continuing inductively, we find that after possibly replacing $E$ by $\sigma_{J'}(E)$ for some $J'\subset\{1,\dots,k\}$, $E$ lies at the intersection of basic hyperplanes.
		But then $E=P$, as we wanted to prove.
	\end{proof}
	It remains to find the extremal rays of $T$.
	\begin{proposition}\label{prop: rays}
		The extremal rays of $T$ are $v_1,\dots,v_k,w_1,\dots,w_k,$ where
		$v_j$ is the $j$th standard basis vector, and $w_j=v_{k+1}-v_j.$
	\end{proposition}
	\begin{proof}
		We consider the homogeneous system of inequalities and use the terminology from above, so a type 0, level $\l$ inequality has the form
		\begin{equation}\label{key}
			x_{j_1}+\dots+x_{j_\l}+z\geq 0,
		\end{equation} and a type $t$, level $\l$ inequality for $1\leq t\leq4$, $0\leq \l\leq k-t$ is
		\begin{equation}\label{key}
			x_{i_1}+\dots+x_{i_t}+2x_{j_1}+\dots+2x_{j_\l}+2z\geq 0.
		\end{equation}
		The extremal rays come from the intersection of hyperplanes which form a system of rank $k$. Let $r=(x_1,\dots,x_k,z)$ be an extremal ray of $T$. We consider two cases, depending on whether $z=0$ or $z>0$.
		
		First, if $z=0$, the type $0$ conditions imply that $x_j\geq 0$ for each $j$. Then we can have equality in any of the inequalities only if $x_j=0$ for some of the $j'$s, and since we need a system of rank $k$, we must have $x_j=0$ for all but one $j\in\{1,\dots,k\}$. This gives $r=v_i$ for $i=1,\dots,k.$
		
		In the second case $z>0$, we can normalize $r$ such that $z=1$. Assume now that $r$ lies on a hyperplane of type $>0$, say
		\begin{equation}\label{key}
			x_{i_1}+\dots+x_{i_t}+2x_{j_1}+\dots+2x_{j_\l}+2z=0.
		\end{equation} We may rewrite the equality and use the type 0 inequalities as
		\begin{equation}\label{key}
			0=\underbrace{x_{i_1}+\dots+x_{i_t}+x_{j_1}+\dots+x_{j_\l}+z}_{\geq 0}+\underbrace{x_{j_1}+\dots+x_{j_\l}+z}_{\geq 0}\geq 0,
		\end{equation} 
		which implies that $r$ lies on the intersection of two other type 0 hyperplanes. It follows that all extremal rays come from intersections of type 0 hyperplanes of rank $k$.
		Let us now consider all the negative coordinates of $r$, assume that these are $x_{i_1},\dots,x_{i_m}$. Then $x_{i_1}+\dots+x_{i_m}+z\geq 0,$ and it is smallest possible among all the expressions $x_{j_1}+\dots+x_{j_\l}+z$, so the only possibility for $r$ to lie on a type $0$ hyperplane is that $x_{i_1}+\dots+x_{i_m}=-1$, in which case $r$ lies only on the type 0 hyperplanes $x_{i_1}+\dots+x_{i_m}+z= 0,$ $x_{i_1}+\dots+x_{i_m}+x_{j_1}+\dots+x_{j_n}+z= 0,$ where $x_{j_1},\dots,x_{j_n}=0$.
		For those hyperplanes to form a system of rank $k$, we need $m=1$, so $x_{i_1}=-1$, and $x_{j}=0$ for each $j\neq i_1$. This way we obtain the extremal rays $w_i$ for $i=1,\dots,k$ and no others.
	\end{proof}
	
	\begin{proof}[Proof of Theorem \ref{thm: region for main theorem}]
		We already proved that $S\subset V$. To prove that $V\subset S,$ it suffices to prove the inclusion of the closures of the base sets of the tube domains \\$T\subset \tilde U=\ch\lz \bigcup\limits_{J\subset\{1,\dots,k\}}\sigma_J(U)\pz$, where $U$ is from Proposition \ref{prop: V-representation of R}. Let $P,Q,v_j,w_j$ be as in propositions \ref{prop: extremal points} and \ref{prop: rays}. Note that for any $J\subset\{1,\dots,k\}$, we have $\sigma_J(Q)=Q$ and for $\tilde\sigma_J$ denoting the linear part of $\sigma_J$, so that $\tilde\sigma_J(v)=\sigma_J(v)+(-\mathbb{1}_{1\in J},\dots,-\mathbb{1}_{k\in J},\frac{|J|}{2})$, we have \begin{equation}\label{key}
			\tilde \sigma_J(v_i)=\begin{cases}
				v_i,&\hbox{if $i\notin J$},\\
				w_i,&\hbox{if $i\in J$,}
			\end{cases}
		\end{equation} for $i=1,\dots,k,$ and $\tilde \sigma_J(v_{k+1})=v_{k+1}$. Proposition \ref{prop: V-representation of R} and \eqref{eqn: convex hull of V-representations} therefore implies that
		\begin{equation}\label{key}
			\tilde U=\ch\lz Q\cup\bigcup_{J\subset\{1,\dots,k\}}\sigma_J(P)\pz+\nonneg(v_1,\dots,v_{k+1},w_1,\dots,w_k)
		\end{equation}
		
		Propositions \ref{prop: extremal points} and \ref{prop: rays} imply that
		\begin{equation}\label{key}
			T=\ch\lz Q\cup\bigcup\limits_{J\subset\{1,\dots,k\}}\sigma_J(P)\pz+\nonneg(v_1,\dots,v_k,w_1,\dots,w_k),
		\end{equation} and $T\subset \tilde U$ follows.
		
	\end{proof}
	
	\section{Proof of Theorem \ref{thm: main theorem}}\label{sec: proof of main theorem}
	We now prove Theorem \ref{thm: main theorem}. Let
	\begin{equation}\label{key}
		\tilde A(s_1,\dots,s_k,w):=\prod_{J\subset\{1,\dots,k\}} \lz w+\sum_{j\in J}s_j-1-\frac{|J|}{2}\pz\cdot A(s_1,\dots,s_k,w).
	\end{equation} Then $\tilde A(s_1,\dots,s_k,w)$ is holomorphic in $R$ by \eqref{eqn: region R}, and by \eqref{eqn: A fe} and \eqref{eqn: A residues} imply that it is also holomorphic in $\sigma_J(R)$, so Bochner's Tube Theorem implies that it is holomorphic in \begin{equation}\label{key}
		\ch\lz\bigcup_{J\subset\{1,\dots,k\}}\sigma_J(R)\pz=S.
	\end{equation} A small technical issue is that the regions are not open and only intersect at the boundary half-planes $s_j=1/2$, but this can be easily dealt with by extending the region $R$ -- note that the restriction $\re(s_j)\geq 1/2$ in \eqref{eqn: basic region 1} was made merely for simplicity, and since $L(s,\chi_{8d})$ is polynomially bounded in $d$ for every $s$, we see that \eqref{eqn: A definition} is absolutely convergent for any complex numbers $s_1,\dots,s_j$ if $\re(w)$ is large enough. For instance, assuming GLH and bounding the individual $L$-functions in \eqref{eqn: A definition again} by the maximum of $d^\epsilon$ or $d^{1/2-\re(s_j)+\epsilon}$, we would obtain the region
	\begin{equation}
		\bigcap\limits_{J\subset\{1,\dots,k\}}\left\{ \re\lz w+\sum_{j\in J}s_j\pz\geq 1+\frac{|J|}{2}\right\}.
	\end{equation} Working with this region would not lead to the technical issue since it would not introduce the condition $\re(s_j)\geq 1/2$, but it would be more complicated and would eventually lead to the same result (since it is essentially the simpler region from \eqref{eqn: basic region 1} transformed under $\sigma_J$).
	
	Theorem \ref{thm: region for main theorem} then implies that $\tilde A(s_1,\dots,s_k,w)$ has a holomorphic continuation to the region from Theorem \ref{thm: main theorem}, and \eqref{eqn: A residues} shows that $A(s_1,\dots,s_k,w)$ has only the stated residues and poles in this region.
	
	To show that $A(s_1,\dots,s_k,w)$ is polynomially bounded in vertical strips, it is enough to prove this in the regions $\sigma_J(R)$ in view of Proposition \ref{prop: extending inequalities} and the discussion below. But this is clear, since it holds for all the expressions \eqref{eqn: A definition}, \eqref{eqn: A after exchanging sums}, \eqref{eqn: A fe} (including the gamma factors), \eqref{eqn: A in terms of A_c}, \eqref{eqn: A_c first expression}, \eqref{eqn: A_c second expression}, \eqref{eqn: fe for A_c}, because it is true for the individual $L$-functions $L(s,\chi_{8d})$, $L\lz w,\leg{\cdot}{n}\pz$ and $K\lz w,\leg{\cdot}{n}\pz.$
	
	\begin{remark}\label{remark: possible uniformity}
		In order to prove the corollaries with some uniformity in the $s_j$ variables, one would need to obtain an explicit polynomial bound for $A(s_1,\dots,s_k,w)$ in vertical strips. It might be possible to (conditionally) prove an analogue of the Lindel\"of bound
		\begin{equation}\label{eqn: explicit bound in vertical strips}
			|A(s_1,\dots,s_k,w)|\ll\prod_{j=1}^k (1+|s_j|)^{\max\{0,1/2-\re(s_j)\}} |w|^{A},
		\end{equation} for some $A>0$, which would imply that Conjecture \ref{con: cfkrs} holds in greater uniformity than stated. An analogous result was proved in \cite[Theorem 1.2]{Cec1}. 
		Note that \eqref{eqn: explicit bound in vertical strips} holds (under GLH) in the individual regions $R$ and $\sigma_J(R)$, but the application of Proposition \ref{prop: extending inequalities} is not straightforward as the bound is not a holomorphic function.

		The same might also apply for the twisted moments and uniformity in the twisting variable; in particular, one might be able to replace the $\l^{1/4+\epsilon}$ in the error term of \cite[Theorem 1]{CoRo} by $\l^\epsilon$.
	\end{remark}
	
	\section{Proof of the corollaries}
	In this section, we prove the corollaries of Theorem \ref{thm: main theorem}.
	
	\begin{proof}[Proof of Corollary \ref{cor: first three moments}]
		Let $k=1,2$ or 3 and let $s_1,\dots,s_k$ be complex numbers with real part $\geq 1/2$. By Perron's formula,
		\begin{equation}\label{key}
			\sumstar_{d\geq 1}f\bfrac{d}{X}L(s_1,\chi_{8d})\dots L(s_k,\chi_{8d})=\frac{1}{2\pi i}\int_{(2)}A(s_1,\dots,s_k,w)X^w\tilde f(w)\ dw.
		\end{equation}
		We shift the integral as far to the left as possible, staying in the region where $A(s_1,\dots,s_k,w)$ is defined, and capture the contribution of the poles. 
		Since $A(s_1,\dots,s_k,w)$ is polynomially bounded in vertical strips and $\tilde f(w)$ decays faster than any polynomial, the horizontal integrals vanish. By Theorem \ref{thm: main theorem}, we may shift the integral to $\re(w)=\max\{1/2,e_k\}+\epsilon,$ where $e_k$ is as in Corollary $\ref{cor: first three moments}$, and bounding the shifted integral trivially yields the error term. The main terms are obtained from the poles and their residues \eqref{eqn: main theorem residues}.
	\end{proof}
	
	\begin{proof}[Proof of Corollary \ref{cor: moments to the right of critical line}]
		Let $\nu,\delta, s_1,\dots,s_k$ be as in Corollary \ref{cor: moments to the right of critical line}. By Perron's formula,
		\begin{equation}\label{key}
			\sumstar_{d\geq 1}f\bfrac{d}{X}L(s_1,\chi_{8d})\dots L(s_k,\chi_{8d})=\frac{1}{2\pi i}\int_{(2)}A(s_1,\dots,s_k,w)X^w\tilde f(w)\ dw.
		\end{equation}
		We shift the integral as far to the left as possible, staying in the region where $A(s_1,\dots,s_k,w)$ is defined, and capture the contribution of the poles. 
		Since $A(s_1,\dots,s_k,w)$ is polynomially bounded in vertical strips and $\tilde f(w)$ decays faster than any polynomial, the horizontal integrals vanish.
		
		To find the region where we can shift the integral, note that since $\re(s_j)>1/2$, the main constraints come from the level 0 conditions. We need to ensure that $\re(w)>1/2$ and
		\begin{equation}\label{key}
			\re(w)>\frac{4+3t}{8}-\sum_{j\leq t}\frac{\re(s_{j_1}+\dots+s_{j_t})}{2}
		\end{equation} for $1\leq t\leq 4$, but
		\begin{equation}\label{key}
			\frac{4+3t}{8}-\sum_{j\leq t}\frac{\re(s_{j_1}+\dots+s_{j_t})}{2}<\frac{4+3t}{8}-\frac{t}{4}-\frac{t\delta}{2}=\frac12+\frac{t}{8}-\frac{t\delta}{2},
		\end{equation} and since $\delta<1/4$, the last expression is always $<1-2\delta.$
		
		We may thus shift the integral to $\re(w)=1-2\delta+\epsilon$, and bounding the shifted integral trivially gives the error term.
		
		We capture the main terms from the pole at $w=1$, which gives the diagonal contribution, and the poles at $w=3/2-s_j$, as $\re(3/2-s_j)>1-2\delta+\nu\delta$, which give the 1-swap terms.
		
		We don't capture any other poles, as all the poles occur at $w=1+\frac{|J|}{2}-\sum\limits_{j\in J}s_j$ for some $J\subset\{1,\dots,k\}$, and \begin{equation}\label{key}
			\re\lz1+\frac{|J|}{2}-\sum\limits_{j\in J}s_j\pz<1-|J|\delta
		\end{equation} is outside of the region of integration if $|J|\geq 2$.
	\end{proof}
	
	\begin{proof}[Proof of Corollary \ref{cor: long Dirichlet polynomials}]
		Let $s_1,\dots,s_k$ be as in Corollary \ref{cor: long Dirichlet polynomials}. Applying Perron's formula twice, we have
		\begin{equation}\label{key}
			\begin{aligned}
				&\sumstar_{d\geq 1}f(d/X)\sum_{n_1,\dots,n_k\geq 1}\frac{\chi_{8d}(n_1\dots n_k)}{n_1^{s_1}\dots n_k^{s_k}}W\bfrac{n_1\dots n_k}{N}\\
				&=\sumstar_{d\geq 1}f(d/X)\sum_{n\geq 1}\chi_{8d}(n)W\bfrac{n}{N}\sum_{n=n_1\dots n_k}\frac{1}{n_1^{s_1}\dots n_k^{s_k}}\\
				&=\bfrac{1}{2\pi i}^2\int_{(2)}\int_{(2)}A(s_1+s,\dots,s_k+s,w)N^s X^w \tilde N(s)\tilde f(w)\ dw\ ds.
			\end{aligned}
		\end{equation}
		We first shift the $s$-integral to $\re(s)=1/4$ and the $w$-integral to $\re(w)=1/2+\epsilon$. To see that we stay within the boundary of the region $S$ where $A(s_1+s,\dots,s_k+s,w)$ is defined, note that $\re(s_j+s)>1/2$, so we only need to verify the level $0$ conditions, which are met.
		
		Bounding the shifted integral trivially gives the error term. While shifting the integrals, we capture the pole at $w=1$ and the poles at $w=3/2-s_j-s,$ which give the 0- and 1-swap terms in the answer. We don't capture any other poles, as all the poles occur at $w=1+\frac{|J|}{2}-\sum\limits_{j\in J}(s_j+s)$ for some $J\subset\{1,\dots,k\}$, and \begin{equation}\label{key}
			\re\lz1+\frac{|J|}{2}-\sum\limits_{j\in J}(s_j+s)\pz=1-\frac{|J|}{4}+O\bfrac{1}{\log X}
		\end{equation} is outside of the region of integration if $|J|\geq 2$ (provided $X$ is large enough).
	\end{proof}


\begin{thebibliography}{200}
		\bibitem{AlBu} B. Albert, A. Bucur, \emph{Counting number fields using multiple Dirichlet series}, preprint (2026), available at 	arXiv:2602.23619.
		\bibitem{Ale} A. D. Alexandrov, \emph{Convex polyhedra}, translated from the 1950 Russian edition by N.S. Dairbekov, S.S. Kutateladze, A. B. Sossinsky, Springer Monographs in Mathematics. Springer, Heidelberg (2005).
		\bibitem{BaCe} S. Baluyot, M. \v Cech, \emph{Multiple Dirichlet series predictions for moments of unitary, symplectic and orthogonal families of $L$-functions}, preprint (2025), available at arXiv:2501.12529.
		\bibitem{BaTu} S. Baluyot, C. Turnage-Butterbaugh, \emph{Twisted $2k$th moments of primitive Dirichlet L-functions: beyond the diagonal}, preprint, available at arXiv:2205.00641.
		
		\bibitem{Bas} A. Basu, \emph{Introduction to convexity}, lecture notes, available online at \url{https://www.ams.jhu.edu/~abasu9/AMS_550-465/notes-without-frills.pdf} (accessed February 2024).
		\bibitem{Cec1} M. \v Cech, \emph{The Ratios conjecture for real Dirichlet characters and multiple Dirichlet series}, Trans. Amer. Math. Soc. 377 (2024), 3487-352.
		\bibitem{Cec2} M. \v Cech, \emph{Mean value of real characters using a double Dirichlet series}, Canad. Math. Bull. 66 (4) (2023), 1135-1151.
		\bibitem{first_moment} M. \v Cech, \emph{First moment of quadratic Dirichlet L-functions with secondary terms}, preprint (2025), available at arXiv:2509.03197.
		\bibitem{CFKRS} J.B. Conrey, D.W. Farmer, J.P. Keating, M.O. Rubinstein and N.C. Snaith. \emph{Integral moments of
			L-functions}, Proc. Lond. Math. Soc. 91 (2005) 33-104.
		\bibitem{CoFa} J.B. Conrey, A. Fazzari, \emph{Averages of long Dirichlet polynomials with modular coefficients}, Mathematika 69 (4) (2023), 1060-1080.
		\bibitem{CoKe1} J. B. Conrey, J. P. Keating, \emph{Moments of zeta and correlations of divisor-sums: I}, Philos.
		Trans. Roy. Soc. A 373: 20140313 (2015).
		
		\bibitem{CoKe2} J. B. Conrey, J. P. Keating, \emph{Moments of zeta and correlations of divisor-sums: II}, Advances in the Theory of Numbers. Vol. 77. Fields Inst. Commun. Fields Inst. Res. Math. Sci., Toronto,
		ON, 2015, pp. 75-85.
		\bibitem{CoKe3} J. B. Conrey, J. P. Keating, \emph{Moments of zeta and correlations of divisor-sums: III}, Indag. Math. (N. S.) 26 (2015), 736-747.
		\bibitem{CoKe4} J. B. Conrey, J. P. Keating, \emph{Moments of zeta and correlations of divisor-sums: IV}, Res. Number Theory 2, Article number: 24 (2016).
		\bibitem{CoKe5} J. B. Conrey, J. P. Keating, \emph{Moments of zeta and correlations of divisor-sums: V}, Proc.
		London Math. Soc. (3) 118 (2019), 729-752.
		\bibitem{CoRo} J. B. Conrey, B. Rodgers, \emph{Averages of quadratic twists of long Dirichlet polynomials}, preprint, available at arXiv:2208.01783.
		\bibitem{DGH} A. Diaconu, D. Goldfeld, J. Hoffstein \textit{Multiple Dirichlet series and moments of zeta and $L$-functions}. Compositio Math. \textbf{139} (2003), 297--360.
		\bibitem{DiWh} A. Diaconu, I. Whitehead, On the third moment of $L(1/2,\chi_d)$ {II}: the number field case. J. Eur. Math. Soc. 23 (2021), no. 6, pp. 2051-2070.
		\bibitem{GaZh1} P. Gao, L. Zhao, \emph{First moment of central values of Hecke $L$-functions with Fixed Order Characters}, preprint, available at arXiv:2306.10726.

		\bibitem{GaZh3} P. Gao, L. Zhao, \emph{Ratios conjecture for quadratic Hecke L-functions in the Gaussian field}, Monatsh. Math. 203 (2024), 63-90.
		\bibitem{GaZh4} P. Gao, L. Zhao, \emph{Ratios conjecture of cubic $L$-functions of prime moduli}, preprint, available at arXiv:2311.08626.
		\bibitem{GaZh5} P. Gao, L. Zhao, \emph{Ratios conjecture for quadratic twist of modular $L$-functions}, preprint, available at arXiv:2304.14600.
		\bibitem{GoHo} D. Goldfeld, J. Hoffstein, \emph{Eisenstein series of 1/2-integral weight and the mean value of real Dirichlet
			L-series,} Inventiones Math. 80 (1985), 185-208.
		\bibitem{HaNg} A. Hamieh and N. Ng. \emph{Mean values of long Dirichlet polynomials with higher divisor coefficients}, Adv. in Math., 410 (B) (2022), 108759. 
		\bibitem{HB} D. R. Heath-Brown, \emph{A mean value estimate for real character sums}. Acta Arith. 72 (3) (1995), 235-275.
		\bibitem{Hor} L. H\"ormander, \emph{An introduction to complex analysis in several variables}, Van Nostrand, Princeton, N.J., 1966.
		\bibitem{HuWe} D. Hug, W. Weil, \emph{Lectures on Convex Geometry}, Grad. Texts in Math. 286, Springer, Cham, 2020.
		\bibitem{Jut} M. Jutila, \emph{On the mean value of $L(\tfrac12,\chi)$ for real characters} Analysis, 1(2):149-161, 1981
		\bibitem{KeSn} J. P. Keating and N. C. Snaith. \emph{Random matrix theory and $L$-functions at $s=\frac12$}, Comm. Math.	Phys., 214(1):91-110, 2000.
		\bibitem{Li} X. Li, \emph{Moments of quadratic twists of modular $L$-functions}, preprint, available at 	arXiv:2208.07343.
		\bibitem{She} Q. Shen, \emph{The fourth moment of quadratic Dirichlet $L$-functions}, Math. Z. 298 (2021), no. 1-2,
		713-745.
		\bibitem{ShSt} Q. Shen, J. Stucky, \emph{The fourth moment of quadratic Dirichlet $L$-Functions II}, preprint, available at arXiv:2402.01497. 
		\bibitem{Sono} K. Sono, \emph{The second moment of quadratic Dirichlet $L$-functions}, J. Number Theory 206 (2020), 194--230. 
		
		\bibitem{Sou} K. Soundararajan, \emph{Nonvanishing of Quadratic Dirichlet L-functions at s=1/2}, Ann. of Math. (2) 152 (2000), 447-488.
		\bibitem{SoYo} K. Soundararajan, M. Young, \emph{The second moment of quadratic twists of modular L-functions},
		J. Eur. Math. Soc. (JEMS) 12 (2010), no. 5, 1097--1116
		
		\bibitem{Tit} E.C. Titchmarsh (revised by D.R. Heath-Brown), \emph{The Theory of the Riemann
			Zeta-Function} (2nd ed.), Oxford University Press (1986)
		
		\bibitem{You} M. P. Young, \emph{The first moment of quadratic Dirichlet $L$-functions}, Acta Arithmetica 138 (2009), no. 1, 73-99.
		
		\bibitem{You2} M. P. Young, \emph{The third moment of quadratic Dirichlet $L$-functions}, Selecta Math. (N.S.) 19 (2013), no. 2, 509-543.
	\end{thebibliography}
\end{document}